\documentclass{amsart}
\usepackage[dvips]{graphicx}
\usepackage{amscd}
\usepackage{pstricks}
\theoremstyle{plain}
\newtheorem{theorem}{Theorem}
\newtheorem{proposition}{Proposition}
\newtheorem{lemma}{Lemma}
\newtheorem{definition}{Definition}
\newtheorem{corollary}{Corollary}

\theoremstyle{remark}
\newtheorem{remark}{Remark}

\numberwithin{equation}{section}
\numberwithin{lemma}{section}

\DeclareMathOperator{\sgn}{Sgn}

\begin{document}

\title[]{A Grobman--Hartman theorem for a differential equation with piecewise constant generalized argument}
\author[]{Manuel Pinto}
\author[]{Gonzalo Robledo}
\address{Departamento de Matem\'aticas, Facultad de Ciencias, Universidad de
  Chile, Casilla 653, Santiago, Chile}
\email{pintoj@uchile.cl,grobledo@uchile.cl}
\thanks{This work was supported by FONDECYT  1120709}
\date{June 2015}

\begin{abstract}
We obtain sufficient conditions ensuring the existence of a uniformly continuous and H\"older continuous homeomorphism between 
the solutions of a
linear system of differential equations with piecewise constant argument of generalized type and the solutions of 
the quasilinear corresponding system. We use a definition (recently introduced by M. Akhmet) of exponential 
dichotomy for those systems combined with technical assumptions on the nonlinear part. Our result generalizes
a previous work of G. Papaschinopoulos.
\end{abstract}

\maketitle

\section{Introduction}
The purpose of this article is to study the strong topological equivalence (see \emph{e.g.}, \cite{Jiang,Jiang1,Palmer,Shi} for definitions) 
between the solutions of the linear differential equation with piecewise constant arguments of generalized type:
\begin{equation}
\label{lineal-efes}
\begin{array}{lcl}
\dot{y}(t)&=&A(t)y(t)+A_{0}(t)y(\gamma(t)),
\end{array}
\end{equation}
and the family of nonlinear systems
\begin{equation} 
\label{sistema1}
\begin{array}{lcl}
\dot{x}(t)&=&A(t)x(t)+A_{0}(t)x(\gamma(t))+f(t,x(t),x(\gamma(t))), 
\end{array}
\end{equation}
provided that (\ref{lineal-efes}) admits an exponential dichotomy, the 
matrices $A(\cdot)$ and $A_{0}(\cdot)$ and 
$f\colon \mathbb{R}\times\mathbb{R}^{n}\times\mathbb{R}^{n}\to \mathbb{R}^{n}$ are such that
\begin{itemize}
 \item[\textbf{(A1)}] There exist positive constants $M$ and $M_{0}$ such that 
 \begin{displaymath}
 \sup\limits_{-\infty<t<+\infty}||A(t)|| \leq M \quad \textnormal{and} \quad \sup\limits_{-\infty<t<+\infty}||A_{0}(t)|| \leq M_{0},
 \end{displaymath}
 where $||\cdot||$ denotes a matrix norm,
 \item[\textbf{(A2)}] there exists a positive constant $\mu$  such that
 \begin{displaymath}
|f(t,x,y)|\leq \mu \quad \textnormal{for any} \quad (t,x,y)\in \mathbb{R}\times\mathbb{R}^{n}\times \mathbb{R}^{n}, 
\end{displaymath} 
where $|\cdot|$ denotes a vector norm.
 \item[\textbf{(A3)}] there exist positive constants $\ell_{1}$ and $\ell_{2}$ such that if $x,x',y,y'\in \mathbb{R}^{n}$
\begin{displaymath}
|f(t,x,y)-f(t,x',y')| \leq \ell_{1}|x-x'|+\ell_{2}|y-y'| \quad \textnormal{for any} \quad t\in \mathbb{R}.
\end{displaymath}
\end{itemize}

The study of systems with piecewise constant arguments begin with Myshkis \cite{Myshkis}, which considers $\gamma(t)=[t]$
(the integer part), this case and other variations were usually known as DEPCA (Differential Equations with 
Piecewise Constant Argument) in the literature. A generalization was made by Akhmet \cite{Akhmet-2011}, which introduces the 
DEPCAG (Differential Equations with Piecewise Constant Generalized Argument) by considering two sequences $\{t_{i}\}_{i\in \mathbb{Z}}$ and $\{\zeta_{i}\}_{i\in \mathbb{Z}}$, which satisfy: 
\begin{itemize}
\item[\textbf{(B1)}] $t_{i}<t_{i+1}$ and $t_{i}\leq \zeta_{i}\leq t_{i+1}$ for any $i\in \mathbb{Z}$,
\item[\textbf{(B2)}] $t_{i}\to \pm \infty$ as $i\to \pm \infty$,
\item[\textbf{(B3)}] $\gamma(t)=\zeta_{i}$ for $t\in [t_{i},t_{i+1})$, 
\item[\textbf{(B4)}] there exists a constant $\theta>0$ such that 
$$
t_{i+1}-t_{i}=\theta_{i}\leq \theta, \quad \textnormal{for any} \quad i\in \mathbb{Z}.
$$
\end{itemize} 

There exists an intensive theoretical research in DEPCAG equations (see, for instance, the monographies \cite{Akhmet-2011,Dai,Wiener}),
which has been accompanied with applications in engineering, life sciences and numerical analysis of ODE--DDE 
systems \cite{Akhmet-2013,Chiu,Gyori,Huang,Nakata,Pinto1,Pinto-Robledo,Seuret,Veloz,Yuan-Favard}.

\subsection{Topological equivalence}
The concept of topological equivalence was introduced by Palmer in \cite{Palmer} and can be seen as a generalization 
of the well known Grobman--Hartman's theorem to a nonautonomous framework.

\begin{definition}
\label{TopEq}
The systems \textnormal{(\ref{lineal-efes})} and \textnormal{(\ref{sistema1})} are topologically equivalent
if there exists a function $H\colon \mathbb{R}\times \mathbb{R}^{n}\to \mathbb{R}^{n}$ with the properties
\begin{itemize}
\item[(i)] For each fixed $t\in \mathbb{R}$, $u\mapsto H(t,u)$ is an homeomorphism of $\mathbb{R}^{n}$,
\item[(ii)] $H(t,u)-u$ is bounded in $\mathbb{R}\times \mathbb{R}^{n}$,
\item[(iii)] if $x(t)$ is a solution of \textnormal{(\ref{sistema1})}, then $H[t,x(t)]$ is a solution 
of \textnormal{(\ref{lineal-efes})},
\end{itemize}
In addition, the function $L(t,u)=H^{-1}(t,u)$ has properties \textnormal{(i)--(iii)} also.
\end{definition}

The concept of strongly topologically equivalence was introduced by Shi and Xiong \cite{Shi}, who realized that, in several
examples of topological equivalence, the maps $u\mapsto H(t,u)$ and $u\mapsto L(t,u)$ could have properties sharper than continuity.

\begin{definition}
\label{StrTopEq}
The systems \textnormal{(\ref{lineal-efes})} and \textnormal{(\ref{sistema1})} are strongly topologically equivalent
if they are topologically equivalent and $H$ and $L$ are uniformly continuous for all $t$.
\end{definition}

\subsection{Exponential dichotomy}
The exponential dichotomy property can be viewed as a generalization of the hiperbolicity property of linear
autonomous systems and plays an important role in the study of linear systems.

\begin{definition}\textbf{(}see \cite{Coppel}\textbf{)}
\label{eidi}
The system
\begin{equation}
\label{ode2}
x'=A(t)x
\end{equation}
has an $\widetilde{\alpha}$--exponential dichotomy if there exists a projection $P$ ($P^{2}=P$) and two constants 
$\tilde{K}\geq 1$,$\tilde{\alpha}>0$ such that $\Phi(t)$, the Cauchy matrix of \textnormal{(\ref{ode2})}, satisfies
 \begin{equation}
\label{edest}
\left\{\begin{array}{rcl}
||\Phi(t)P\Phi^{-1}(s)||\leq \tilde{K}e^{-\tilde{\alpha}(t-s)} & \textnormal{if} & t\geq s\\
||\Phi(t)(I-P)\Phi^{-1}(s)||\leq \tilde{K}e^{-\tilde{\alpha}(s-t)} &\textnormal{if}& s > t.
\end{array}\right.
\end{equation}
\end{definition}

There are not a univoque definition of exponential dichotomy in a DEPCAG framework. The main dificulty is that
the transition matrix $Z(t,\tau)$ of (\ref{lineal-efes}) can be constructed only when certain technical conditions
are satisfied (see section 2). We will consider two definitions:
\begin{definition} \textnormal{(}Akhmet \cite{Akhmet-2012,Akhmet-2013}\textnormal{)}
\label{ED1}
The linear DEPCAG \textnormal{(\ref{lineal-efes})} has an $\alpha$--exponential dichotomy on $(-\infty,\infty)$ if there exists a projection 
$P$ and some constants $K\geq 1$ and $\alpha>0$, such that its transition matrix $Z(t,s)$ verifies
\begin{equation}
\label{ED-Ineg}
||Z_{p}(t,s)|| \leq Ke^{-\alpha|t-s|} 
\end{equation}
where $Z_{p}(t,s)$ is defined by
\begin{equation}
\label{FG}
Z_{p}(t,s)=\left\{\begin{array}{rcl}
Z(t,0)PZ(0,s) & \textnormal{if} & t \geq s\\
-Z(t,0)\{I-P\}Z(0,s) &\textnormal{if}& s > t.
\end{array}\right.
\end{equation}
\end{definition}

\begin{definition}
\label{EDP}
 The linear DEPCAG \textnormal{(\ref{lineal-efes})} has an exponential dichotomy on $(-\infty,\infty)$ if the system of difference equations
\begin{equation}
\label{difference}
y_{n+1}=Z(t_{n+1},t_{n})y_{n} 
\end{equation}
has a discrete exponential dichotomy, which means that there exists a projection $\hat{P}$, $\hat{K}\geq 1$ and $0<r<1$ such that $Y_{n}$, the
Cauchy matrix of \textnormal{(\ref{difference})} verifies
\begin{displaymath}
\left\{\begin{array}{rcl}
||Y_{n}\hat{P}Y_{m}^{-1}||\leq \hat{K}r^{n-m} & \textnormal{if} & n\geq m\\
||Y_{n}\{I-\hat{P}\}Y_{m}^{-1}||\leq \hat{K}r^{m-n} &\textnormal{if}& m>n.
\end{array}\right.
\end{displaymath}
\end{definition}

\begin{remark}
\label{about-dic}
Notice that:
\begin{itemize}
\item[i)] Definition \ref{ED1} has been recently introduced by Akhmet in \cite{Akhmet-2012,Akhmet-2013} in order to study
the existence of almost periodic solutions of almost periodic perturbations of (\ref{lineal-efes}). Definition \ref{EDP} is 
employed in \cite{Castillo} with similar purposes. It is important to note that Definition \ref{ED1} is oriented 
to a global treatement of (\ref{lineal-efes}) while Definition \ref{EDP} allows the reduction to (\ref{difference}).\\
\item[ii)] A particular but distinguished case of Definition \ref{EDP} restricted to $\gamma(t)=[t]$ was previously introduced by 
Papaschinopoulos \cite{Papaschinopoulos,Papaschinopoulos2}.\\
\item[iii)] Definitions \ref{ED1} and \ref{EDP} are independent and none implies the other. A deeper study about the relationship 
between definitions above remains to be done. Some preliminar comparative examples are presented in \cite{Castillo}.
\end{itemize}
\end{remark}

\subsection{Background and developments}
The seminal paper of Palmer \cite{Palmer} proves that if (\ref{ode2}) has an exponential dichotomy (\ref{edest}) and the perturbed system
\begin{equation}
\label{ode1}
x'=A(t)x+f(t,x),
\end{equation}
satisfies
\begin{equation}
\label{estimaciones}
|f(t,x)|\leq \tilde{\mu} \quad \textnormal{and} \quad |f(t,x_{1})-f(t,x_{2})|\leq \tilde{\ell}|x_{1}-x_{2}| \quad 
\textnormal{for all}\quad t,x,x_{1},x_{2},
\end{equation}
then (\ref{ode2}) and (\ref{ode1}) are topologically equivalent provided that $2\tilde{\ell} \tilde{K} \leq \tilde{\alpha}$.

Palmer's result of topological equivalence has been generalized in several directions: ordinary differential 
equations \cite{Barreira,Jiang,Jiang1,Shi}, difference equations \cite{BV-06,Castaneda,Kurzweil-Papas,Papas2},
impulsive equations \cite{Pinto-Fenner,Xia-2013} and time-scales systems \cite{Potzche,Xia}.

In a DEPCA framework, there exists a result of topological equivalence obtained by G. Papaschinopoulos 
\cite[Proposition 1]{Papaschinopoulos} for the special case $\gamma(t)=[t]$ by following the lines of the Palmer's work
and introducing its \emph{ad-hoc} definition of exponential dichotomy for (\ref{lineal-efes}).

This work generalizes the topological equivalence result of \cite{Papaschinopoulos} in several directions. Firstly, we consider a general piecewise constant argument of advanced/delayed type. Secondly, we obtain conditions for strongly and H\"older strongly topological equivalence. Thirdly, instead of Papaschinopoulos's definition of exponential dichotomy of (\ref{lineal-efes}), we use Definition \ref{ED1}, which allows a global treatment and considers limit cases that cannot be trated by the Papaschinopoulos's definition. More technical generalizations will be explained later. 

\subsection{Outline} Section 2 introduces technical notation, recalls the variation
of parameters formula presented in \cite{Pinto} and states a 
result (Theorem \ref{boundeness}) about existence and uniqueness of bounded solutions for bounded perturbations of (\ref{lineal-efes}).
Section 3 states the two main results (Theorems \ref{intermedio} and \ref{intermedio2}) of stronlgy topological equivalence. Sections 4 and 5 state technical intermediate 
results. The proof of the main results is finished in section 6. 
\section{Technical preliminaries}
In order to make the article self--contained, we will recall some previous notation and results obtained in \cite{Pinto}.
\begin{definition}\cite{Akhmet-2013,Wiener}
A continuous function $u(t)$ is solution of \textnormal{(\ref{lineal-efes})} or \textnormal{(\ref{sistema1})}
if:
\begin{itemize}
\item[(i)] The derivative $u'(t)$ exists at each point $t\in \mathbb{R}$ with the possible exception of the 
points $t_{i}$, $i\in \mathbb{Z}$, where the one side derivatives exists;
\item[(ii)] The equation is satisfied for $u(t)$ on each interval $(t_{i},t_{i+1})$, and it holds for the right
derivative of $u(t)$ at the points $t_{i}$. 
\end{itemize}
\end{definition}

Without loss of generality, we will assume that the Cauchy matrix of (\ref{ode2}) satisfies $\Phi(0)=I$. As usual, the transition matrix related to $A(t)$ will be denoted by $\Phi(t,s)=\Phi(t)\Phi^{-1}(s)$.

In \cite{Akhmet-2011,Pinto}, the following $n\times n$ matrices are introduced:
\begin{equation}
\label{simbolo-0}
J(t,\tau)= I+\int_{\tau}^{t}\Phi(\tau,s)A_{0}(s)\,ds,
\end{equation}
\begin{equation}
\label{simbolo}
E(t,\tau)= \Phi(t,\tau)+\int_{\tau}^{t}\Phi(t,s)A_{0}(s)\,ds=\Phi(t,\tau)J(t,\tau).
\end{equation}

Given a set of $n\times n$ matrices $\mathcal{Q}_{k}$ ($k=1,\ldots,m$), we will consider the product in the backward and
 forward sense as follows:
\begin{displaymath}
\prod\limits_{k=1}^{\leftarrow m}\mathcal{Q}_{k}=\left\{
\begin{array}{cl}
\mathcal{Q}_{m}\cdots \mathcal{Q}_{2}\mathcal{Q}_{1} &\quad \textnormal{if} \quad m\geq 1\\
I & \quad \textnormal{if} \quad m<1.
\end{array}\right.
\end{displaymath}
and
\begin{displaymath}
\prod\limits_{k=1}^{\rightarrow m}\mathcal{Q}_{k}=\left\{
\begin{array}{cl}
\mathcal{Q}_{1}\mathcal{Q}_{2}\cdots \mathcal{Q}_{m} &\quad \textnormal{if} \quad m\geq 1\\
I & \quad \textnormal{if} \quad m<1.
\end{array}\right.
\end{displaymath}

\subsection{Notation and facts related to the sequences $\{t_{i}\}$ and $\{\zeta_{i}\}$}
The following notation will be useful:
\begin{enumerate}
 \item[$\bullet$] For any $k\in \mathbb{Z}$, we define $I_{k}=[t_{k},t_{k+1})$, $I_{k}^{+}=[t_{k},\zeta_{k}]$ and
$I_{k}^{-}=[\zeta_{k},t_{k+1})$.
 \item[$\bullet$] For any $t\in \mathbb{R}$, we define $i(t)\in \mathbb{Z}$ as the unique integer
 such that $t\in I_{i}=[t_{i},t_{i+1})$.
 \item[$\bullet$] The number of the terms of the sequence $\{t_{i}\}$ contained in the interval $(\tau,t)$ will be denoted by $i(\tau,t)$.
 \item[$\bullet$] For any $k\in \mathbb{Z}$ and any matrix $t\mapsto Q(t)\in M_{n}(\mathbb{R})$, we define the numbers:
\begin{displaymath}
\rho_{k}^{+}(Q)=\exp\Big(\int_{t_{k}}^{\zeta_{k}}|Q(s)|\,ds\Big), \quad \textnormal{and} \quad\rho_{k}^{-}(Q)=\exp\Big(\int_{\zeta_{k}}^{t_{k+1}}|Q(s)|\,ds\Big).
\end{displaymath}
\end{enumerate}

Some examples of functions $\gamma(t)$ and its corresponding sequences $\{t_{k}\}$ and $\{\zeta_{k}\}$ 
satisfying \textbf{(B1)--(B4)} are summarized in the following 
table (see \cite{Wiener} for details):

\begin{center}
\begin{tabular}{|c|c|c|l|l|}
\hline
$\gamma(t)$ & $\{t_{k}\}$ & $\{\zeta_{k}\}$ & Restrictions & Comments \\
\hline
$[t]$ & $k$ &  $k$ &  & completely delayed\\
\hline
$[t-j]$ & $k$ & $k-j$ & $j\in \mathbb{Z}^{+}$ & completely delayed\\
\hline
$[t+j]$ & $k$ & $k+j$ & $j\in \mathbb{Z}^{+}$ & completely advanced\\
\hline
$[t+1/2]$ & $k$ & $k+1/2$&  & advanced/delayed\\
\hline
$2[(t+1)/2]$ & $2k$& $2k+1$ & & advanced/delayed\\
\hline
$\alpha h[t/(\alpha h)]$ & $k\alpha h$& $k\alpha h$ & $\alpha>0$,$h>0$ & completely delayed\\
\hline
$m[(t+j)/m]$ &  $mk-j$ & $mk$ & $m>j>0$ &  advanced/delayed\\
\hline
\end{tabular}
\end{center}
\bigskip

It is interesting to point out that the last two examples are functions $t\mapsto \gamma(t)$ 
employed in DEPCAG equations while the previous ones are classical examples used in DEPCA equations.
The qualitative difference is that, in the first examples, the sequences $\{t_{k}\}$ and $\{\zeta_{k}\}$
are strictly determined, while in last cases they are dependent of the parameters $\alpha$ and $m$ respectively, which induce 
$\alpha$--parameter (resp. $m$--parameter) dependent families of sequences $\{t_{k}\}$ and $\{\zeta_{k}\}$. 

\begin{lemma}
\label{EUUG}
For any $s$ and $t$, it follows that
\begin{equation}
\label{borne}
|\gamma(s)-t|\leq \theta + |t-s|,
\end{equation}
where $\theta$ is the same stated in \textnormal{\textbf{(B4)}}.
\end{lemma}
\begin{proof}
As $s\in [t_{i(s)},t_{i(s)+1})$, it follows that $\gamma(s)=\zeta_{i(s)}$. Now \textbf{(B1)} implies that
\begin{displaymath}
t_{i(s)}-t_{i(s)+1}\leq \zeta_{i(s)}-t_{i(s)+1}<\gamma(s)-s<\zeta_{i(s)}-t_{i(s)}<t_{i(s)+1}-t_{i(s)} 
\end{displaymath}
and \textbf{(B4)} implies that $|\gamma(s)-s|\leq \theta$.

Finally, (\ref{borne}) follows from $|\gamma(s)-t|\leq |\gamma(s)-s|+|s-t|$.
\end{proof}

\subsection{Complementary assumptions about $A$ and $A_{0}$}
Throughout this article, we will asume that
\begin{itemize}
 \item[\textbf{(C)}] There exists $\nu^{+}>0$ and $\nu^{-}>0$ such that the 
matrices $A(t)$ and $A_{0}(t)$ satisfy the properties:
\begin{equation}
\label{prop-mat2}
\sup\limits_{k\in \mathbb{Z}}\rho_{k}^{+}(A)\ln\rho_{k}^{+}(A_{0})\leq \nu^{+}<1 \quad \textnormal{and} \quad
\sup\limits_{k\in \mathbb{Z}}\rho_{k}^{-}(A)\ln\rho_{k}^{-}(A_{0})\leq \nu^{-}<1.
\end{equation}
\end{itemize}

Notice that \textbf{(A1)} and \textbf{(B4)} imply that
 \begin{equation}
\label{prop-mat1}
\rho(A)=\sup\limits_{k\in \mathbb{Z}}\rho_{k}^{+}(A)\rho_{k}^{-}(A)<+\infty.
\end{equation}

An important consequence of \textbf{(C)} is the following result:
\begin{lemma}\cite[Lemma 4.3]{Pinto}
\label{est-mt}
If \textnormal{(\ref{prop-mat2})} is verified, it follows that
\begin{displaymath}
|\Phi(t,s)|\leq \rho(A) \quad \textnormal{for any} \quad t,s\in I_{i}.
\end{displaymath}
and $J(t,s)$ is nonsingular for any $t,s\in I_{i}$. 
\end{lemma}


\subsection{Variation of parameters formula}
Throughout the rest of this section, it will be assumed that \textbf{(A)},\textbf{(B)} and \textbf{(C)} are satisfied.

A distinguished feature of DEPCAG systems is that their solutions could be noncontinuable in several cases. In this context, the 
condition \textbf{(C)} is introduced in \cite{Pinto} in order to provide sufficient conditions ensuring the continuability
of the solutions of (\ref{lineal-efes}) to $(-\infty,+\infty)$. Furthermore, condition \textbf{(C)} and Lemma \ref{est-mt}
imply that $J(t,s)$ and $E(t,s)$ are nonsingular for any $t,s\in I_{i}$, which allow to construct the transition matrix 
for (\ref{lineal-efes}) and to derive the variation of parameters formula.

\begin{proposition}\cite[p.239]{Pinto}
\label{pitaron}
For any $t\in I_{j}$,$\tau\in I_{i}$, the solution of \textnormal{(\ref{lineal-efes})} with $z(\tau)=\xi$
is defined by
$$
z(t)=Z(t,\tau)\xi,
$$
where $Z(t,\tau)$ is defined by
\begin{equation}
\label{transition2}
\begin{array}{rcl}
Z(t,\tau)&=&E(t,\zeta_{j})E(t_{j},\zeta_{j})^{-1}\prod\limits_{k=i+2}^{\leftarrow j}E(t_{k},\gamma(t_{k-1}))E(t_{k-1},\gamma(t_{k-1}))^{-1}\\\\
&& E(t_{i+1},\gamma(\tau))E(\tau,\gamma(\tau))^{-1},
\end{array}
\end{equation}
when $t>\tau$ and by
\begin{equation}
\label{transition2-inv}
\begin{array}{rcl}
Z(t,\tau)&=&E(t,\zeta_{j})E(t_{j},\zeta_{j})^{-1}\prod\limits_{k=i+2}^{\rightarrow j}E(t_{k},\gamma(t_{k}))E(t_{k},\gamma(t_{k-1}))^{-1}\\\\
&& E(t_{i},\gamma(\tau))E(\tau,\gamma(\tau))^{-1},
\end{array}
\end{equation}
when $t<\tau$.
\end{proposition}

\begin{remark}
\label{about-z}
A direct consequence of Proposition \ref{pitaron} is that the operator $Z(\cdot,\cdot)$ verifies
\begin{equation}
\label{MT1}
Z(t,\tau)Z(\tau,s)=Z(t,s) \quad \textnormal{and} \quad Z(t,s)=Z(s,t)^{-1}.
\end{equation}

In addition, by using the facts 
\begin{displaymath}
E(\tau,\tau)=I \quad \textnormal{and} \quad \frac{\partial E}{\partial t}(t,\tau)=A(t)E(t,\tau)+A_{0}(t)
\end{displaymath}
combined with Proposition \ref{pitaron}, we can deduce that:
\begin{equation}
\label{MT2}
\frac{\partial Z}{\partial t}(t,\tau)=A(t)Z(t,\tau)+A_{0}(t)Z(\gamma(t),\tau).
\end{equation}
\end{remark}




\begin{proposition}[Th. 3.1, \cite{Pinto}]
\label{vp}
For any $j> i$, $t\in I_{j}$ and $\tau\in I_{i}$, the solution of 
\begin{equation}
\label{perturbado}
\begin{array}{lcl}
\dot{x}(t)&=&A(t)x(t)+A_{0}(t)x(\gamma(t))+g(t),
\end{array}
\end{equation}
with $z(\tau)=\xi$
is defined by
\begin{displaymath}
\begin{array}{rcl}
x(t)&=&\displaystyle Z(t,\tau)\xi+\int_{\tau}^{\zeta_{i}}Z(t,\tau)\Phi(\tau,s)g(s)\,ds+\sum\limits_{r=i+1}^{j}\int_{t_{r}}^{\zeta_{r}}Z(t,t_{r})\Phi(t_{r},s)g(s)\,ds\\\\
&&\displaystyle +\sum\limits_{r=i}^{j-1}\int_{\zeta_{r}}^{t_{r+1}}Z(t,t_{r+1})\Phi(t_{r+1},s)g(s)\,ds\\\
&&\displaystyle +\sgn(t-\zeta_{j})\int_{\min\{\zeta_{j},t\}}^{\max\{\zeta_{j},t\}}\hspace{-0.3cm}\Phi(t,s)g(s)\,ds,
\end{array}
\end{displaymath}
when $\tau\in I_{i}^{+}=[t_{i},\zeta_{i})$.
\end{proposition}

It is important to emphasize that, when we consider any interval $I_{k}=[t_{k},t_{k+1})$, we have
the corresponding system of difference equations
\begin{displaymath}
\begin{array}{rcl}
x(t_{n+1})&=&\displaystyle Z(t_{n+1},t_{n})x(t_{n})+\int_{t_{n}}^{\zeta_{n}}Z(t_{n+1},t_{n})\Phi(t_{n},s)g(s)\,ds\\
&&\displaystyle +\int_{\zeta_{n}}^{t_{n+1}}\Phi(t_{n+1},s)g(s)\,ds,
\end{array}
\end{displaymath}
which plays a key role to obtain the solution of (\ref{perturbado}). This non-homogeneous difference equation justifies Definition
\ref{EDP}. The most studied case is $t_{n}=n$, 
that arises when $\gamma(t)=[t]$.

\begin{lemma}
\label{UBS}
If the linear DEPCAG \textnormal{(\ref{lineal-efes})} has an $\alpha$--exponential dichotomy on $(-\infty,\infty)$, then
the unique solution bounded on $(-\infty,+\infty)$ is the null solution.
\end{lemma}

\begin{proof}
By following the lines of Coppel \cite{Coppel}, let us 
note that (\ref{ED-Ineg}) is equivalent to:  
\begin{displaymath}
\begin{array}{rl}
|Z(t,0)P\nu|\leq Ke^{-\alpha(t-s)}|Z(s,0)P\nu|  & \textnormal{for} \quad t\geq s\\\\
|Z(t,0)(I-P)\nu| \leq Ke^{-\alpha(s-t)}|Z(s,0)(I-P)\nu| &  \textnormal{for} \quad t<s.
\end{array}
\end{displaymath} 
for any arbirtary $\nu \in \mathbb{R}^{n}$. Let us assume that $P$ has rank $k$, then, the first inequality
says that there is a $k$--dimensional vector space of initial conditions, such that it corresponding solutions
converge to $0$ when $t\to +\infty$ (and are divergent when $s\to -\infty$). The second 
inequality says that there is a complementary $(n-k)$--dimensional
space, whose corresponding solutions are divergent when $s\to +\infty$ (and converge to $0$ when $t\to -\infty$).
The conclusion follows easily from those properties.
\end{proof}

Now, let us define the Green function corresponding to (\ref{lineal-efes}) in the interval $(-\infty,\infty)$:
\begin{definition}
\label{Green2}
Given $t\in (\zeta_{j},t_{j+1})$ and $Z_{p}(t,\tau)$ introduced in \textnormal{(\ref{FG})}, let us define
\begin{displaymath}
\widetilde{G}(t,s)=\left\{\begin{array}{rcl}
Z_{p}(t,t_{r})\Phi(t_{r},s)&\quad \textnormal{if} \quad & s\in [t_{r},\zeta_{r}) \quad \textnormal{for any} \quad r\in \mathbb{Z},\\
Z_{p}(t,t_{r+1})\Phi(t_{r+1},s)& \quad \textnormal{if} \quad & s\in [\zeta_{r},t_{r+1}) \quad \textnormal{for any} \quad r\in \mathbb{Z}\setminus\{j\},\\
\Phi(t,s) &\quad \textnormal{if} \quad & s\in [\zeta_{j},t),\\
0&\quad \textnormal{if} \quad & s\in [t,t_{j+1}),\\
\end{array}\right.
\end{displaymath}
and if $t\in [t_{j},\zeta_{j}]$
\begin{displaymath}
\widetilde{G}(t,s)=\left\{\begin{array}{rcl}
Z_{p}(t,t_{r})\Phi(t_{r},s)&\quad \textnormal{if} \quad & s\in [t_{r},\zeta_{r}) \quad \textnormal{for any} \quad r\in \mathbb{Z}\setminus\{j\},\\
Z_{p}(t,t_{r+1})\Phi(t_{r+1},s)& \quad \textnormal{if} \quad & s\in [\zeta_{r},t_{r+1}) \quad \textnormal{for any} \quad r\in \mathbb{Z},\\
0 &\quad \textnormal{if} \quad & s\in [t_{j},t),\\
-\Phi(t,s)&\quad \textnormal{if} \quad & s\in [t,\zeta_{j}),\\
\end{array}\right.
\end{displaymath}
\end{definition}

It is important to observe that $\widetilde{G}$ takes into account delayed and advanced intervals.

\begin{proposition}
\label{cotas-green}
If the DEPCAG \textnormal{(\ref{lineal-efes})} has an $\alpha$--exponential dichotomy \textnormal{(\ref{ED-Ineg})}, then $\widetilde{G}$ satisfies 
\begin{equation}
\label{cota-imp}
|\widetilde{G}(t,s)|\leq K\rho^{*} e^{-\alpha|t-s|}, \quad \textnormal{where} \quad \rho^{*}=\rho(A)e^{\alpha \theta}.
\end{equation}
\end{proposition}

\begin{proof}
Without loss of generality, let us assume that $\zeta_{j}<t<t_{j+1}$. If $s\notin [t_{j},t_{j+1}]$, there exists
$r\neq j$ such that either $s\in [t_{r},\zeta_{r}]$  or $s\in [\zeta_{r},t_{j+1}]$.

Firstly, if $s\in [t_{r},\zeta_{r}]$ and $j>r$, we have that $t>t_{r}$. This fact, combined with Lemma \ref{est-mt}, eq.(\ref{ED-Ineg}) and  
Definition \ref{Green2} imply
\begin{displaymath}
\begin{array}{rcl}
|\widetilde{G}(t,s)|&=&|Z_{p}(t,t_{r})\Phi(t_{r},s)|\\ 
&\leq& Ke^{-\alpha(t-t_{r})}\rho(A)\\
&\leq& Ke^{-\alpha(t-s)}\rho(A)e^{\alpha \theta}.
\end{array}
\end{displaymath}

Secondly, if $s\in [t_{r},\zeta_{r}]$ and $j<r$, we have that $t\leq t_{r}\leq s$. As before, we can deduce that 
\begin{displaymath}
\begin{array}{rcl}
|\widetilde{G}(t,s)|&\leq& Ke^{-\alpha(t_{r}-t)}\rho(A)\\
&\leq& Ke^{-\alpha(t_{r}-s)}\rho(A)\\
&\leq& Ke^{-\alpha(t-s)}\rho^{*}.
\end{array}
\end{displaymath}

The reader can obtain similar estimations in the case $s\in [\zeta_{r},t_{j+1}]$. Finally, if $s\in I_{j}$, by using $K\geq 1$
combined with Lemma \ref{est-mt}, we can deduce that
\begin{displaymath}
\begin{array}{rcl}
|\widetilde{G}(t,s)|&\leq & \rho(A)\\
&\leq & K\rho(A)e^{\alpha|t-s|}e^{-\alpha|t-s|}\\
&\leq & Ke^{-\alpha|t-s|}\rho^{*},
\end{array}
\end{displaymath} 
and the Lemma follows.
\end{proof}

\begin{remark}
\label{tetp}
Notice that if $\theta$ is arbitrarily small, then $\rho^{*}$ is arbitrarily close to one and equation (\ref{cota-imp}) is close to
$$
|\widetilde{G}(t,s)|\leq Ke^{-\alpha|t-s|},
$$
which is the estimation of the Green's function in the ODE case.

\end{remark}

\begin{theorem}
\label{boundeness}
If DEPCAG \textnormal{(\ref{lineal-efes})} has an $\alpha$--exponential dichotomy and the series
\begin{equation}
\label{serie1}
\sum\limits_{r=-\infty}^{k}PZ(0,t_{r})\int_{t_{r}}^{\zeta_{r}}\Phi(t_{r},s)\,ds, \quad
\sum\limits_{r=-\infty}^{k}PZ(0,t_{r+1})\int_{\zeta_{r}}^{t_{r+1}}\Phi(t_{r+1},s)\,ds, 
\end{equation}
and
\begin{equation}
\label{serie2}
\sum\limits_{r=k}^{+\infty}(I-P)Z(0,t_{r})\int_{t_{r}}^{\zeta_{r}}\hspace{-0.3cm}\Phi(t_{r},s)\,ds, \hspace{0.1cm}
\sum\limits_{r=k}^{+\infty}(I-P)Z(0,t_{r+1})\int_{\zeta_{r}}^{t_{r+1}}\hspace{-0.5cm}\Phi(t_{r+1},s)\,ds, 
\end{equation}
are absolutely convergent for any integer $k$, then for each bounded function $t\mapsto g(t)$, the 
system \textnormal{(\ref{perturbado})} has a unique solution bounded 
on $\mathbb{R}$, defined by
$$
x_{g}^{*}(t)=\int_{-\infty}^{\infty}\widetilde{G}(t,s)g(s)\,ds
$$
and the map $g\mapsto x_{g}$ is Lipschitz satisfying
$$
|x_{g}^{*}|_{\infty}\leq \frac{2K\rho^{*}}{\alpha}|g|_{\infty},
$$
\end{theorem}

\begin{proof}
Without loss of generality, let us assume that $0\in [t_{i},\zeta_{i})$ and $t\in [\zeta_{j},t_{j+1})$ with $j>i$. 

\noindent\emph{Step 1:} We will prove that
\begin{displaymath}
\begin{array}{rcl}
x_{g}^{*}(t)&=&\displaystyle \sum\limits_{r=-\infty}^{j}\int_{t_{r}}^{\zeta_{r}}Z(t,0)PZ(0,t_{r})\Phi(t_{r},s)g(s)\,ds\\\\
&&\displaystyle+\sum\limits_{r=-\infty}^{j-1}\int_{\zeta_{r}}^{t_{r+1}}Z(t,0)PZ(0,t_{r+1})\Phi(t_{r+1},s)g(s)\,ds\\\\
&&\displaystyle -\sum\limits_{r=j+1}^{+\infty}\int_{t_{r}}^{\zeta_{r}}Z(t,0)(I-P)Z(0,t_{r})\Phi(t_{r},s)g(s)\,ds\\\\
&&\displaystyle -\sum\limits_{r=j+1}^{+\infty}\int_{\zeta_{r}}^{t_{r+1}}Z(t,0)(I-P)Z(0,t_{r+1})\Phi(t_{r+1},s)g(s)\,ds+\int_{\zeta_{j}}^{t}\Phi(t,s)g(s)\,ds,
\end{array}
\end{displaymath}
is a bounded solution of (\ref{perturbado}). Indeed, by using eq.(\ref{MT2}) combined with $\int_{\zeta_{j}}^{\zeta_{j}}\Phi(\zeta_{j},s)g(s)\,ds=0$, 
it is easy to see that $t\mapsto x^{*}(t)$ is solution of (\ref{perturbado}). On the other hand, a careful reading of Definition \ref{Green2} shows that
$$
x_{g}^{*}(t)=\int_{-\infty}^{+\infty}\widetilde{G}(t,s)g(s)\,ds,
$$ 
and the boundedness follows from Proposition \ref{cotas-green}.

\noindent\emph{Step 2:} We will prove that $x_{g}^{*}(t)$ is the unique bounded solution of (\ref{perturbado}). Indeed, let $t\mapsto x(t)$
be a bounded solution. By using Proposition \ref{vp} with $\tau=0$, we have that
\begin{displaymath}
\begin{array}{rcl}
x(t)&=&\displaystyle Z(t,0)x(0)+\int_{0}^{\zeta_{i}}Z(t,0)\Phi(0,s)g(s)\,ds+\sum\limits_{r=i+1}^{j}\int_{t_{r}}^{\zeta_{r}}Z(t,t_{r})\Phi(t_{r},s)g(s)\,ds\\\\
&&\displaystyle +\sum\limits_{r=i}^{j-1}\int_{\zeta_{r}}^{t_{r+1}}Z(t,t_{r+1})\Phi(t_{r+1},s)g(s)\,ds+\int_{\zeta_{j}}^{t}\Phi(t,s)g(s)\,ds,
\end{array}
\end{displaymath}
which can be written as
\begin{displaymath}
\begin{array}{rcl}
x(t)&=&\displaystyle Z(t,0)\Big\{x(0)+\int_{0}^{\zeta_{i}}\Phi(0,s)g(s)\,ds+\sum\limits_{r=i+1}^{j}P\int_{t_{r}}^{\zeta_{r}}Z(0,t_{r})\Phi(t_{r},s)g(s)\,ds\\\\
&&\displaystyle +\sum\limits_{r=i}^{j-1}P\int_{\zeta_{r}}^{t_{r+1}}Z(0,t_{r+1})\Phi(t_{r+1},s)g(s)\,ds\Big\}+\int_{\zeta_{j}}^{t}\Phi(t,s)g(s)\,ds\\\\
&&+Z(t,0)\Big\{\sum\limits_{r=i+1}^{j}(I-P)\int_{t_{r}}^{\zeta_{r}}Z(0,t_{r})\Phi(t_{r},s)g(s)\,ds\\\\
&&+\sum\limits_{r=i}^{j-1}(I-P)\int_{\zeta_{r}}^{t_{r+1}}Z(0,t_{r+1})\Phi(t_{r+1},s)g(s)\,ds\Big\}.
\end{array}
\end{displaymath}

By using (\ref{serie1})--(\ref{serie2}), we have that
\begin{displaymath}
\begin{array}{rlc}
\sum\limits_{r=i+1}^{j}P\int_{t_{r}}^{\zeta_{r}}Z(0,t_{r})\Phi(t_{r},s)g(s)\,ds&=&\sum\limits_{r=-\infty}^{j}P\int_{t_{r}}^{\zeta_{r}}Z(0,t_{r})\Phi(t_{r},s)g(s)\,ds-\\\\
&&\sum\limits_{r=-\infty}^{i}P\int_{t_{r}}^{\zeta_{r}}Z(0,t_{r})\Phi(t_{r},s)g(s)\,ds,
\end{array}
\end{displaymath}
\begin{displaymath}
\begin{array}{rcl}
\sum\limits_{r=i}^{j-1}P\int_{\zeta_{r}}^{t_{r+1}}Z(0,t_{r+1})\Phi(t_{r+1},s)g(s)\,ds&=&\sum\limits_{r=-\infty}^{j-1}P\int_{\zeta_{r}}^{t_{r+1}}Z(0,t_{r+1})\Phi(t_{r+1},s)g(s)\,ds-\\\\
&&\sum\limits_{r=-\infty}^{i-1}P\int_{\zeta_{r}}^{t_{r+1}}Z(0,t_{r+1})\Phi(t_{r+1},s)g(s)\,ds
\end{array}
\end{displaymath}
and
\begin{displaymath}
\begin{array}{rcl}
\sum\limits_{r=i+1}^{j}(I-P)\int_{t_{r}}^{\zeta_{r}}Z(0,t_{r})\Phi(t_{r},s)g(s)\,ds&=&\sum\limits_{r=i+1}^{+\infty}(I-P)\int_{t_{r}}^{\zeta_{r}}Z(0,t_{r})\Phi(t_{r},s)g(s)\,ds-\\\\
&&\sum\limits_{r=j+1}^{+\infty}(I-P)\int_{t_{r}}^{\zeta_{r}}Z(0,t_{r})\Phi(t_{r},s)g(s)\,ds.
\end{array}
\end{displaymath}

Moreover, notice that
\begin{displaymath}
\begin{array}{l}
\sum\limits_{r=i}^{j-1}(I-P)\int_{\zeta_{r}}^{t_{r+1}}Z(0,t_{r+1})\Phi(t_{r+1},s)g(s)\,ds=\\\\
\sum\limits_{r=i,r\neq j}^{+\infty}(I-P)\int_{\zeta_{r}}^{t_{r+1}}Z(0,t_{r+1})\Phi(t_{r+1},s)g(s)\,ds-
\sum\limits_{r=j+1}^{+\infty}(I-P)\int_{\zeta_{r}}^{t_{r+1}}Z(0,t_{r+1})\Phi(t_{r+1},s)g(s)\,ds,
\end{array}
\end{displaymath} 
and we can see that the bounded solution $x(t)$ can be written as follows
\begin{displaymath}
x(t)=Z(t,0)\{x(0)+x_{1}+x_{2}\}+x_{g}^{*}(t),
\end{displaymath}
where 
\begin{displaymath}
\begin{array}{rcl}
x_{1}&=&\int_{0}^{\zeta_{i}}\Phi(0,s)g(s)-\sum\limits_{r=-\infty}^{i}P\int_{t_{r}}^{\zeta_{r}}Z(0,t_{r})\Phi(t_{r},s)g(s)\,ds\\\\
&&-\sum\limits_{r=-\infty}^{i-1}P\int_{\zeta_{r}}^{t_{r+1}}Z(0,t_{r+1})\Phi(t_{r+1},s)g(s)\,ds
\end{array}
\end{displaymath}
and
\begin{displaymath}
\begin{array}{rcl}
x_{2}&=&\sum\limits_{r=i+1}^{+\infty}(I-P)\int_{t_{r}}^{\zeta_{r}}Z(0,t_{r})\Phi(t_{r},s)g(s)\,ds\\\\
&&+\sum\limits_{r=i,r\neq j}^{+\infty}(I-P)\int_{\zeta_{r}}^{t_{r+1}}Z(0,t_{r+1})\Phi(t_{r+1},s)g(s)\,ds.
\end{array}
\end{displaymath}

As $t\mapsto x_{g}^{*}(t)$ is a bounded solution of (\ref{perturbado}), we have that,
$$
x(t)-x_{g}^{*}(t)=Z(t,0)\{x(0)+x_{1}+x_{2}\}
$$
is a bounded solution of (\ref{lineal-efes}). Finally, Lemma \ref{UBS} implies that $x(t)=x_{g }^{*}(t)$ and the uniqueness follows.
\end{proof}

\begin{remark}
Theorem \ref{boundeness} generalizes a classical result in the ODE case (see \emph{e.g.}, \cite{Coppel,Massera}) 
and has been previously proved by Akhmet and Yilmaz in \cite{Akhmet-2012,Akhmet-2013}. We point out that
our proof was stated in \cite{Pinto-sub} and has some technical differences: we follow a constructive approach 
to deduce the bounded solution, we consider 
the intervals $[t_{r},\zeta_{r})$ and 
$[\zeta_{r},t_{r+1})$ instead of $(\zeta_{r},\zeta_{r+1})$ and we work with different upper bounds of the transition matrix 
$Z(t_{r+1},t_{r})$. 
\end{remark}



\begin{remark}
The convergence of series (\ref{serie1})--(\ref{serie2}) can be ensured by imposing additional properties to the sequence $\{t_{r}\}_{r}$. 
For example, by $\alpha$--exponential dichotomy (\ref{ED-Ineg}) combined with $Z(0,0)=I$ and Lemma \ref{est-mt}, we conclude that
\begin{displaymath}
\begin{array}{rcl}
\sum\limits_{r=k}^{+\infty}\Big|(I-P)Z(0,t_{r+1})\int_{\zeta_{r}}^{t_{r+1}}\Phi(t_{r+1},s)\Big|\,ds &\leq &K\rho(A)\sum\limits_{r=k}^{+\infty}e^{-\alpha |t_{r+1}|},\\
\end{array}
\end{displaymath}
and the second series of (\ref{serie2}) converges if the series $S_{n}=\sum\limits_{k=r}^{n}e^{-\alpha|t_{r+1}|}$ ($n>k$)
is convergent. Now, the convergence of $S_{n}$ can be ensured in several cases. For example, if there exists $\bar{\theta}>0$ such that
$$
\bar{\theta}\leq t_{r+1}-t_{r} \quad \textnormal{for any} \quad r\in \mathbb{Z},
$$
we have that the series $S_{n}$ is dominated by a geometric one. On the other hand, it is straightforward 
to see that, if there exists $C>0$ such that
\begin{displaymath}
|t_{r+2}|-|t_{r+1}|\geq  C \quad \textnormal{or} \quad |t_{r}|\geq C|r| \quad \textnormal{for any} \quad r>R
\end{displaymath}
for any $R$ arbitrarly large, then
\begin{displaymath}
\limsup\limits_{r\to +\infty}e^{-\alpha(|t_{r+2}|-|t_{r+1}|)}<1 \quad \textnormal{or} \quad \limsup\limits_{r\to +\infty}e^{-\frac{|t_{r+1}|}{r}}<1,
\end{displaymath}
which implies the convergence of the series by the ratio test or the radical test respectively.
\end{remark}

Throughout this paper, we will assume that (\ref{serie1})--(\ref{serie2}) are convergent.

\section{Main Results}

\begin{theorem}
\label{intermedio}
If \textnormal{(\ref{lineal-efes})} has a transition matrix $Z(t,0)$ satisfying the exponential 
dichotomy \textnormal{(\ref{ED-Ineg})}, conditions \textnormal{\textbf{(A)}}, \textnormal{\textbf{(B)}} 
and \textnormal{\textbf{(C)}} are satisfied
and
\begin{equation}
\label{FPT}
2(\ell_{1}+\ell_{2})K\rho^{*}<\alpha,
\end{equation}
\begin{equation}
\label{schema0}
F_{1}(\theta)(M_{0}+\ell_{2})\theta=v<1,\quad \textnormal{with} \quad F_{1}(\theta)=\frac{e^{(M+\ell_{1})\theta}-1}{(M+\ell_{1})\theta},
\end{equation}
\begin{equation}
\label{schema0B}
F_{0}(\theta)M_{0}\theta=\tilde{v}<1, \quad \textnormal{with} \quad F_{0}(\theta)=\frac{e^{M\theta}-1}{M\theta},
\end{equation}
then \textnormal{(\ref{lineal-efes})} and \textnormal{(\ref{sistema1})} are strongly topologically equivalent.
\end{theorem}

\begin{theorem}
\label{intermedio2}
If \textnormal{(\ref{lineal-efes})} has a transition matrix $Z(t,0)$ satisfying the exponential 
dichotomy \textnormal{(\ref{ED-Ineg})}, conditions \textnormal{\textbf{(A)}}, \textnormal{\textbf{(B)}}, \textnormal{\textbf{(C)}}, 
\textnormal{(\ref{FPT})}--\textnormal{(\ref{schema0B})} are satisfied and
\begin{equation}
\label{alfa}
\alpha<M+\min\Big\{\ell_{1}+\frac{M_{0}+\ell_{2}}{1-v}e^{(M+\ell_{1})\theta},\frac{M_{0}}{1-\tilde{v}}e^{M\theta}\Big\},
\end{equation}
then the systems \textnormal{(\ref{lineal-efes})} and \textnormal{(\ref{sistema1})} are H\"older strongly 
topologically equivalent, namely, there exists constants $C_{1}>1$,$D_{1}>1$,$C_{2}\in (0,1)$ and $D_{2}\in (0,1)$ such that the maps
$H$ and $L$ are H\"older continuous in the sense:
\begin{displaymath}
|H(t,\xi)-H(t,\xi')|\leq C_{1}|\xi-\xi'|^{C_{2}} \quad \textnormal{and} \quad
|L(t,\nu)-H(t,\nu')|\leq D_{1}|\nu-\nu'|^{D_{2}}
\end{displaymath}
for any couple $(\xi,\xi')$ and $(\nu,\nu')$ verifying $|\xi-\xi'|<1$ and $|\nu-\nu'|<1$.
\end{theorem}

\begin{remark}
As we stated in the introduction, we generalize Papaschinopoulos's result \cite[Proposition 1]{Papaschinopoulos} in several ways:
\begin{itemize}
\item[i)] Theorems \ref{intermedio} and \ref{intermedio2} consider a generic piecewise constant argument including the particular
delayed case $\gamma(t)=[t]$, 
\item[ii)] We obtain results sharper than topological equivalence, namely, strongly and H\"older topological equivalence,
\item[iii)] We use a recently introduced definition of exponential dichotomy,
\item[iv)] Our results don't need to assume that (\ref{ode2}) has the exponential dichotomy (\ref{edest}) and allow limit cases 
as $A(t)=0$ for any $t\in \mathbb{R}$, 
\item[v)] The smallness of $A_{0}(\cdot)$ is not always necessary as in \cite{Papaschinopoulos}, for example 
a threshold between $\theta$ and $M_{0}$ ensuring $v<1$ can be constructed. 
\end{itemize}
\end{remark}

\begin{remark}
\label{Aboutcon}
Some comments about the conditions:
\begin{itemize}
\item[i)] Inequality (\ref{FPT}) is reminiscent of the contractivity condition stated by Palmer in \cite{Palmer}. Notice that
if $\theta=0$ (\emph{i.e.}, $\rho^{*}=1$) and $\ell_{2}=0$, then (\ref{FPT}) becomes the Palmer's condition $2\ell_{1}K<\alpha$.
\item[ii)] Inequalities (\ref{schema0}) and (\ref{schema0B}) can be verified in several cases. For example, when 
$\theta$ is arbitrarily small. Indeed, notice that if $\theta\to 0^{+}$, then $F_{0}(\theta),F_{1}(\theta)\to 1$ 
and $v\approx (M_{0}+\ell_{2})\theta<1$ (resp. $\tilde{v} \approx M_{0}\theta<1$). 
\item[iii)] In the section 2 of \cite{Chiu-2010}, it is proved that the inequality (\ref{schema0}) implies 
the existence and uniqueness of the solutions of (\ref{sistema1}). Indeed, it will be useful to denote 
by $x(t,\tau,\xi)$ as the unique solution of (\ref{sistema1}) passing through $\xi$ at $t=\tau$. By uniqueness of solutions of (\ref{sistema1}), 
we know that
\begin{equation}
\label{unicidad1}
x\big(s,t,x(t,\tau,\xi)\big)=x\big(s,\tau,\xi\big). 
\end{equation}
\item[iv)] Inequality (\ref{alfa}) is related with the H\"older continuity in the classical 
strongly topological equivalence literature (see \emph{e.g.},\cite{Shi}). In addition, it is always satisfied when $\alpha<M$.
\end{itemize}
\end{remark}

The first byproduct states that strongly topological equivalence is an equivalence relation 
since the composition of homeomorphisms is an homeomorphism and its proof is left to the reader:
\begin{corollary}
\label{corolario}
Let us consider the system
\begin{equation} 
\label{sistema14}
\begin{array}{lcl}
\dot{x}(t)&=&A(t)x(t)+A_{0}(t)x(\gamma(t))+h(t,x(t),x(\gamma(t))), 
\end{array}
\end{equation}
where $A$,$A_{0}$ and $h$ satisfy \textnormal{\textbf{(A)}} and $\gamma(\cdot)$
satisfies \textnormal{\textbf{(B)}}. If the assumptions of Theorem \textnormal{\ref{intermedio}} 
(resp. Theorem \textnormal{\ref{intermedio2}}) are satisfied,
then \textnormal{(\ref{sistema1})} and \textnormal{(\ref{sistema14})} 
are strongly topologically equivalent (resp. H\"older strongly topologically equivalent).
\end{corollary}

In the limit case $A_{0}(t)=0$, we have that assumption \textbf{(C)} is always verified since $\ln\rho_{k}^{+}(A_{0})=\ln\rho_{k}^{-}(A_{0})=0$.
In addition, the linear DEPCAG system (\ref{lineal-efes}) becomes the ODE system
(\ref{ode2}). Finally, we can see that $J(t,\tau)=I$, $E(t,\tau)=Z(t,\tau)=\Phi(t,\tau)$ and 
the Green function $\widetilde{G}(t,s)$ becomes:
\begin{displaymath}
G(t,s)=
\left\{\begin{array}{rcl}
\Phi(t)P\Phi^{-1}(s) & \textnormal{if} & t\geq s\\
-\Phi(t)(I-P)\Phi^{-1}(s) &\textnormal{if}& s>t.
\end{array}\right.
\end{displaymath}

Now, it is easy to prove the following result:
\begin{corollary}
\label{intermedio-cor}
If the system \textnormal{(\ref{ode2})} has a Cauchy matrix $\Phi(t)$ satisfying the $\tilde{\alpha}$--exponential 
dichotomy \textnormal{(\ref{edest})}, $A_{0}(t)=0$ for any $t$, conditions \textnormal{\textbf{(A)}} and \textnormal{\textbf{(B)}} 
are satisfied in this context and
\begin{equation}
\label{FPTcor}
2(\ell_{1}+\ell_{2})\tilde{K}<\tilde{\alpha},
\end{equation}
\begin{equation}
\label{schema0cor}
F_{1}(\theta)\ell_{2}\theta=v_{0}<1,
\end{equation}
then the systems \textnormal{(\ref{sistema1})} and \textnormal{(\ref{ode2})} are strongly topologically equivalent. In addition,
if $M>\tilde{\alpha}$, then the systems \textnormal{(\ref{sistema1})} and \textnormal{(\ref{ode2})} are H\"older 
strongly topologically equivalent.
\end{corollary}

Finally, if $A(t)=0$, we have that (\ref{lineal-efes})--(\ref{sistema1}) becomes
\begin{equation}
\label{lineal-efes-cor2}
\begin{array}{lcl}
\dot{y}(t)&=&A_{0}(t)y(\gamma(t)),
\end{array}
\end{equation}
\begin{equation} 
\label{sistema1-cor2}
\begin{array}{lcl}
\dot{x}(t)&=&A_{0}(t)x(\gamma(t))+f(t,x(t),x(\gamma(t))).
\end{array}
\end{equation}

In this context, the reader can verify that $\Phi(t,\tau)=I$ and 
\begin{displaymath}
J(t,\tau)=E(t,\tau)=I+\int_{\tau}^{t}A_{0}(s)\,ds.
\end{displaymath}

In addition, $A(t)=0$ modify the corresponding definitions of $Z(t,s)$ and $\widetilde{G}(t,s)$ 
with $\rho^{*}=e^{\alpha\theta}$ and it is easy to prove:
\begin{corollary}
\label{intermedio-cor2}
If \textnormal{(\ref{lineal-efes-cor2})} has a transition matrix $Z(t,0)$ satisfying the exponential 
dichotomy \textnormal{(\ref{ED-Ineg})}, conditions \textnormal{\textbf{(A)}} and \textnormal{\textbf{(B)}} 
are satisfied and
\begin{equation}
\label{FPT-cor2}
2(\ell_{1}+\ell_{2})Ke^{\alpha\theta}<\alpha,
\end{equation}
\begin{equation}
\label{schema0-cor2}
\tilde{F}_{1}(\theta)(M_{0}+\ell_{2})\theta=\tilde{v}_{0}<1, \quad \textnormal{with} \quad \tilde{F}_{1}(\theta)=\frac{e^{\ell_{1}\theta}-1}{\ell_{1}\theta}
\end{equation}
\begin{equation}
\label{schema0B-cor2}
(M_{0}+\ell_{2})\theta=\tilde{u}_{0}<1,
\end{equation}
then \textnormal{(\ref{lineal-efes-cor2})} and \textnormal{(\ref{sistema1-cor2})} are strongly topologically 
equivalent. In addition, if 
\begin{displaymath}
\alpha< \min\Big\{\ell_{1}+\frac{M_{0}+\ell_{2}}{1-\tilde{v}_{0}}e^{\ell_{1}\theta},\frac{M_{0}}{1-\tilde{u}_{0}}\Big\},
\end{displaymath}
then the systems \textnormal{(\ref{sistema1})} and \textnormal{(\ref{ode2})} are H\"older strongly topologically equivalent.
\end{corollary}
\begin{proof}
We only need to prove that \textbf{(C)} is satisfied with $A(t)=0$. Indeed, notice that $\rho(A)=1$ combined with 
\textbf{(A1)} and (\ref{schema0B-cor2}) imply that
\begin{displaymath}
\ln \rho_{k}^{\pm}(A_{0})\leq M_{0}\theta<\tilde{u}_{0}<1 
\end{displaymath}
and (\ref{prop-mat2}) follows.

\end{proof}

\begin{remark}
It is interesting to see that if $\theta \to 0^{+}$, then the (step) 
function $\gamma(t)$ converges uniformly to the identity function. This case is important in numerical 
approximation for solutions of differential delay equations (see \emph{e.g.}, \cite{Gyori} for details). Moreover, 
the authors are working in the problem of the approximation of the solutions of the ODE systems 
\begin{equation}
\label{ode11}
y'=A_{0}(t)y
\end{equation}
and 
\begin{equation}
\label{ode22}
x'=A_{0}(t)x+f(t,x,x),
\end{equation} 
uniformly on $(-\infty,+\infty)$ by solutions of (\ref{lineal-efes-cor2})--(\ref{sistema1-cor2}) 
when $\theta\to 0^{+}$ and some preliminar results are presented in \cite{Gonzalez}. In this framework, these expected approximation results combined with corollaries
\ref{intermedio-cor} and \ref{intermedio-cor2} could help to deduce and generalize (by an alternative approach) 
the classical Palmer's result \cite{Palmer} about topological equivalence between (\ref{ode11}) and (\ref{ode22}). Notice that conditions 
\textbf{(A)},\textbf{(B)},\textbf{(C)} and inequalities (\ref{FPT-cor2})--(\ref{schema0B-cor2}) "converge`` to those stated in Palmer's article. See 
Remarks \ref{tetp} and \ref{Aboutcon}.
\end{remark}

\section{Some Lemmas}
Throughout this section, we will assume that the system \textnormal{(\ref{lineal-efes})} has a transition 
matrix $Z(t,0)$ satisfying the exponential dichotomy \textnormal{(\ref{ED-Ineg})}. 

\begin{lemma}
\label{lemme-0}
For any solution $x(t,\tau,\xi)$ of \textnormal{(\ref{sistema1})} passing
through $\xi$ at $t=\tau$, there exists a unique bounded solution $t\mapsto \chi(t;(\tau,\xi))$ of
\begin{equation}
\label{auxiliar} 
\begin{array}{lcl}
\dot{z}(t)&=&A(t)z(t)+A_{0}(t)z(\gamma(t))-f(t,x(t,\tau,\xi),x(\gamma(t),\tau,\xi)).
\end{array}
\end{equation}
\end{lemma}

\begin{proof}
By using Theorem \ref{boundeness} with $g(t)=-f(t,x(t,\tau,\xi),x(\gamma(t),\tau,\xi))$, we have that  
\begin{displaymath}
\chi(t;(\tau,\xi))=-\int_{-\infty}^{\infty}\widetilde{G}(t,s)f(s,x(s,\tau,\xi),x(\gamma(s),\tau,\xi))\,ds
\end{displaymath}
is the unique bounded solution of (\ref{auxiliar}). In addition, \textbf{(A2)} implies 
that $|\chi(t;(\tau,\xi))|\leq 2K\rho^{*} \mu \alpha^{-1}$.
\end{proof}

\begin{remark}
\label{uniqueness}
By uniqueness of solutions of (\ref{sistema1}) and equation (\ref{unicidad1}) with $s=t$ and $s=\gamma(t)$, we know that
\begin{displaymath}
x\big(t,t,x(t,\tau,\xi)\big)=x\big(t,\tau,\xi\big) \quad \textnormal{and}\quad  x\big(\gamma(t),t,x(t,\tau,\xi)\big)=x\big(\gamma(t),\tau,\xi\big),
\end{displaymath} 
this fact implies that system (\ref{auxiliar}) can be written as
\begin{displaymath}
\begin{array}{lcl}
\dot{z}(t)&=&A(t)z(t)+A_{0}(t)z(\gamma(t))\\\\
&&-f(t,x(t,t,x(t,\tau,\xi)),x(\gamma(t),t,x(t,\tau,\xi)))
\end{array}
\end{displaymath}
and Lemma \ref{lemme-0} implies that
\begin{equation}
\label{uniqueness1}
\chi(t;(\tau,\xi))=\chi(t;(t,x(t,\tau,\xi))).
\end{equation} 
\end{remark}

\begin{lemma}
\label{lemme-0-b}
For any solution $y(t,\tau,\nu)$ of \textnormal{(\ref{lineal-efes})} passing through $\nu$ at $t=\tau$, there 
exists a unique bounded solution $t\mapsto \vartheta(t;(\tau,\nu))$ of
\begin{equation}
\label{auxiliar2} 
\begin{array}{lcl}
\dot{w}(t)&=&A(t)w(t)+A_{0}(t)w(\gamma(t))\\\\
&&+f(t,y(t,\tau,\nu)+w(t),y(\gamma(t),\tau,\nu)+w(\gamma(t))). 
\end{array}
\end{equation}
\end{lemma}

\begin{proof}
Let $BC$ be the Banach space of bounded and continuous functions $\varphi\colon \mathbb{R}\to \mathbb{R}^{n}$ with 
supremun norm. By Theorem \ref{boundeness}, we 
know that the map $\Gamma \colon BC \to BC$:
\begin{displaymath}
\Gamma \varphi(t)=\int_{-\infty}^{\infty}\widetilde{G}(t,s)f(s,y(s,\tau,\nu)+\varphi(s),y(\gamma(s),\tau,\nu)+\varphi(\gamma(s))) \,ds,
\end{displaymath}
is well defined. Now, notice that \textbf{(A3)} implies 
\begin{displaymath}
\begin{array}{rcl}
|\Gamma \varphi(t) - \Gamma \phi(t)| &\leq & \int_{\mathbb{R}}|\widetilde{G}(t,s)|\{\ell_{1}|\varphi(s)-\phi(s)|+\ell_{2}|\varphi(\gamma(s))-\phi(\gamma(s))|\}\,ds\\\\
                                         &\leq & \displaystyle \frac{2K\rho^{*}}{\alpha}(\ell_{1}+\ell_{2})||\varphi-\phi||,
\end{array}
\end{displaymath}
and (\ref{FPT}) implies that $\Gamma$ is  a contraction, having a unique fixed point satisfying
\begin{equation}
\label{fijo}
\begin{array}{l}
\vartheta(t;(\tau,\nu))=\\\\
\displaystyle \int_{-\infty}^{+\infty}\widetilde{G}(t,s)f(s,y(s,\tau,\nu)+\vartheta(s;(\tau,\nu)),y(\gamma(s),\tau,\nu)+\vartheta(\gamma(s);(\tau,\nu)))\,ds
\end{array}
\end{equation}
and the reader can easily verify that is a bounded solution of (\ref{auxiliar}).
\end{proof}

\begin{remark}
\label{uniquenessB}
Similarly as in Remark \ref{uniqueness}, the reader can verify that
\begin{equation}
\label{uniqueness1B}
\vartheta(t;(\tau,\nu))=\vartheta(t;(t,y(t,\tau,\nu))).
\end{equation} 
\end{remark}

\begin{lemma}
\label{lema1}
There exists a unique function 
$H\colon \mathbb{R}\times \mathbb{R}^{n}\to \mathbb{R}^{n}$, satisfying:
\begin{enumerate}
 \item[(i)] $H(t,x)-x$ is bounded in $\mathbb{R}\times \mathbb{R}^{n}$,
 \item[(ii)] For any solution $t\mapsto x(t)$ of \textnormal{(\ref{sistema1})}, then $t\mapsto H[t,x(t)]$ 
is a solution of \textnormal{(\ref{lineal-efes})} satisfying
\begin{equation}
\label{cota1}
|H[t,x(t)]-x(t)|\leq 2\mu K\rho^{*}\alpha^{-1}
\end{equation}
\end{enumerate}
\end{lemma}

\begin{proof}
The proof will be decomposed in several steps.\\
\noindent\textit{Step 1) Existence of $H$:} Let us define the 
function $H\colon \mathbb{R}\times\mathbb{R}^{n}\to \mathbb{R}^{n}$
as follows
\begin{equation}
\label{homeo-00}
\begin{array}{rcl}
H(t,\xi)&=&\xi+\chi(t;(t,\xi))\\\\
&=&\displaystyle \xi-\int_{-\infty}^{\infty}\widetilde{G}(t,s)f(s,x(s,t,\xi),x(\gamma(s),t,\xi))\,ds
\end{array}
\end{equation}
and \textbf{(A2)} implies $|H(t,\xi)-\xi|\leq 2\mu K\rho^{*} \alpha^{-1}$.

By replacing $(t,\xi)$ by $(t,x(t,\tau,\xi))$ in (\ref{homeo-00}), we have that
\begin{displaymath}
\begin{array}{rcl}
H[t,x(t,\tau,\xi)]&=&x(t,\tau,\xi)+\chi(t;(t,x(t,\tau,\xi)))\\\\
\end{array}
\end{displaymath}
Now, by (\ref{uniqueness1}), we have 
\begin{equation}
\label{homeosofit}
H[t,x(t,\tau,\xi)]=x(t,\tau,\xi)+\chi(t;(\tau,\xi))
\end{equation}
or equivalently
\begin{equation}
\label{HomeH1}
\begin{array}{rcl}
H[t,x(t,\tau,\xi)]&=&\displaystyle x(t,\tau,\xi)-\int_{-\infty}^{\infty}\widetilde{G}(t,s)f(s,x(s,\tau,\xi),x(\gamma(s),\tau,\xi))\,ds.
\end{array}
\end{equation} 

Finally, it is easy to verify that $t\mapsto H[t,x(t,\tau,\xi)]$ is solution of (\ref{lineal-efes}).

\noindent\textit{Step 2) Uniqueness of $H$:} Let us suppose that there exists another map $\widetilde{H}$ satisfying 
properties (i) and (ii), this implies that
$\widetilde{H}[t,x(t,\tau,\xi)]$ is solution of (\ref{lineal-efes}) and
\begin{displaymath}
\hat{z}(t,\xi)=\widetilde{H}[t,x(t,\tau,\xi)]-x(t,\tau,\xi)
\end{displaymath}
is a bounded solution of (\ref{auxiliar}). Nevertheless, as (\ref{auxiliar}) has a unique bounded solution, we can 
conclude that $\hat{z}(t)=\chi(t;(\tau,\xi))$
and (\ref{homeosofit}) implies that
\begin{displaymath}
\widetilde{H}[t,x(t,\tau,\xi)]=x(t,\tau,\xi)+\chi(t;(\tau,\xi))=H[t,x(t,\tau,\xi)].
\end{displaymath}
\end{proof}

\begin{lemma}
\label{lema2}
There exists a unique function 
$L\colon \mathbb{R}\times \mathbb{R}^{n}\to \mathbb{R}^{n}$, satisfying:
\begin{enumerate}
 \item[(i)] $L(t,y)-y$ is bounded in $\mathbb{R}\times \mathbb{R}^{n}$,
 \item[(ii)] For any solution $t\mapsto y(t)$ of \textnormal{(\ref{lineal-efes})}, we have that 
 $t\mapsto L[t,y(t)]$ is a  solution of \textnormal{(\ref{sistema1})} verifying 
\begin{equation}
\label{cota2}
|L[t,y(t)]-y(t)|\leq 2\mu K\rho^{*} \alpha^{-1}.
\end{equation}
\end{enumerate}
\end{lemma}

\begin{proof}
The existence and uniqueness of the function $L$ satisfying (i)--(ii) can be proved in a similar way. Indeed, $L$ 
is defined by
$$
L(t,\nu)=\nu+\vartheta(t;(t,\nu)),
$$
where
\begin{displaymath}
\vartheta(t;(t,\nu))=\int_{-\infty}^{\infty}\widetilde{G}(t,s)f(s,y(s,t,\nu)+\vartheta(s;(t,\nu)),y(\gamma(s),t,\nu)+\vartheta(\gamma(s);(t,\nu)))\,ds
\end{displaymath}

As before, by using (\ref{uniqueness1B}), for $y(t)=y(t,\tau,\nu)$ we can define
\begin{equation}
\label{HomoL}
L[t,y(t)]=y(t,\tau,\nu)+\vartheta(t;(t,y(t,\tau,\nu)))=y(t,\tau,\nu)+\vartheta(t;(\tau,\nu)),
\end{equation}

It will be useful to describe $L[t,y(t)]$ as follows
\begin{equation}
\label{HomeL1}
\begin{array}{rcl}
L[t,y(t)]&=&\displaystyle y(t)+\int_{-\infty}^{\infty}\widetilde{G}(t,s)f(s,L[s,y(s)],L[\gamma(s),y(\gamma(s))])\,ds.\\\\
\end{array}
\end{equation}
\end{proof}

\begin{lemma}
\label{composition}
For any solution $x(t)$ of \textnormal{(\ref{sistema1})} and $y(t)$ of \textnormal{(\ref{lineal-efes})} with fixed $t$, it follows that
\begin{displaymath}
L[t,H[t,x(t)]]=x(t) \quad \textnormal{and} \quad H[t,L[t,y(t)]]=y(t).
\end{displaymath}
\end{lemma}

\begin{proof}
We will prove only the first identity. The other one can be deduced similarly and is given for the reader. 

Let $t\mapsto x(t)=x(t,\tau,\xi)$ be a solution of (\ref{sistema1}). By using 
Lemma \ref{lema1}, we know that $H[t,x(t)]$ is solution of (\ref{lineal-efes}). Moreover, by Lemma \ref{lema2}, we can see that
$t\mapsto J[t,x(t)]=L[t,H[t,x(t)]]$ is 
solution of (\ref{sistema1}). Notice that
\begin{displaymath}
\begin{array}{rcl}
J[t,x(t)] & = & H[t,x(t)]+\vartheta(t;(t,H[t,x(t)]))\\
\end{array}
\end{displaymath}
where $t\mapsto \vartheta(t;(t,H[t,x(t)]))$ is the unique bounded solution of the system 
\begin{displaymath}
\begin{array}{rcl}
\dot{w}(t)&=&A(t)w(t)+A_{0}(t)w(\gamma(t))\\\\
&&+f(t,H[t,x(t)]+w(t),H[\gamma(t),x(\gamma(t))]+w(\gamma(t))). 
\end{array}
\end{displaymath}

By using Lemma \ref{lema2} with $H[t,x(t)]$ instead of $y(t)$, we have that
\begin{displaymath}
\begin{array}{rcl}
J[t,x(t)]&=&\displaystyle H[t,x(t)]+\int_{-\infty}^{\infty}\widetilde{G}(t,s)f(s,J[s,x(s)],J[\gamma(s),x(\gamma(s))])\,ds.\\\\
\end{array}
\end{displaymath}

Upon inserting (\ref{HomeH1}) in the identity above, we have that
\begin{displaymath}
J[t,x(t)]-x(t)=\int_{\mathbb{R}}\widetilde{G}(t,s)\{f(s,J[s,x(s)],J[\gamma(s),x(\gamma(s))])-f(s,x(s),x(\gamma(s)))\}\,ds, 
\end{displaymath}
which implies the inequality
\begin{displaymath}
|J[t,x(t)]-x(t)|\leq \frac{2K\rho^{*}}{\alpha}\big(\ell_{1}+\ell_{2}\big)|J[\cdot,x(\cdot)]-x(\cdot)|_{\infty}
\end{displaymath}
and (\ref{FPT}) implies that 
\begin{displaymath}
J[t,x(t)]=L[t,H[t,x(t)]]=x(t).
\end{displaymath}
\end{proof}

The reader can notice (see also Definition \ref{TopEq}) that the notation $H[\cdot,\cdot]$ and $L[\cdot,\cdot]$
is reserved to the case when $H$ and $L$ are respectively defined on solution of (\ref{sistema1}) and (\ref{lineal-efes}).

\begin{lemma}
\label{invertibilidad}
For any fixed $t$ and any couple $(\xi,\nu)\in \mathbb{R}^{n}\times\mathbb{R}^{n}$, it follows that
\begin{equation}
\label{inve1}
L(t,H(t,\xi))=\xi 
\end{equation}
and
\begin{equation}
\label{inve2}
H(t,L(t,\nu))=\nu.
\end{equation}
\end{lemma}
\begin{proof}
By using Lemma \ref{composition}, we have that
\begin{displaymath}
L[t,H[x(t,\tau,\xi)]=x(t,\tau,\xi) \quad \textnormal{for any} \quad t\in \mathbb{R}.
\end{displaymath}

Now, if we consider the particular case  $\tau=t$, we obtain (\ref{inve1}). The identity (\ref{inve2}) can be deduced
similarly.
\end{proof}

\begin{remark}
\label{tofin}
Notice that the maps $\xi  \mapsto H(t,\xi)$ and $\nu \mapsto L(t,\nu)$ satisfy properties
(ii) and (iii) of Definition \ref{TopEq}, which is a consequence of Lemmatas \ref{lema1}--\ref{composition}. In addition,
Lemma \ref{invertibilidad} says that $u\mapsto L(t,u)=H^{-1}(t,u)$ for any $t\in \mathbb{R}$. In consequence, the last step is to prove
the uniform continuity of the maps, which will be made in the next two sections. 
\end{remark}

\section{Continuity with respect to initial conditions}

The following result generalizes the classical Gronwall's inequality to the DEPCAG framework:
\begin{proposition}\textnormal{(}Gronwall's type inequality, \cite[Lemma 2.1]{Chiu-2010}\textnormal{)}
\label{LGKS}
Let $u$,$\tilde{\eta}_{i}\colon \mathbb{R}\to [0,+\infty)$ $i=1,2$ be continuous functions 
and $\tilde{C}>0$. Suppose that for all $t\geq \tau$, the inequality
\begin{displaymath}
u(t)\leq \tilde{C} + \int_{\tau}^{t}\{\tilde{\eta}_{1}(s)u(s)+\tilde{\eta}_{2}(s)u(\gamma(s))\}\,ds 
\end{displaymath}
holds. If
\begin{displaymath}
w=\sup\limits_{i\in \mathbb{N}}\int_{t_{i}}^{\zeta_{i}}\tilde{\eta}_{2}(s)e^{\int_{s}^{\zeta_{i}}\tilde{\eta}_{1}(r)\,dr}\,ds<1,
\end{displaymath}
then for any $t\geq \tau$ it follows that
\begin{displaymath}
\displaystyle
u(t)\leq \tilde{C}\exp\Big(\int_{\tau}^{t}\tilde{\eta}_{1}(s)\,ds+\frac{1}{1-w}\int_{\tau}^{t}\Big[\tilde{\eta}_{2}(s)e^{\int_{t_{i}(s)}^{\gamma(s)}\tilde{\eta}_{1}(r)\,dr}\Big]\,ds\Big). 
\end{displaymath}
\end{proposition}

Similarly as in an ODE context, the Gronwall's inequality is a key tool in the proof of continuity with respect to the initial conditions:
\begin{lemma}
\label{CRCI}
Let $t\mapsto x(t,\tau,\xi)$ and $t \mapsto x(t,\tau,\xi')$ be the solutions of \textnormal{(\ref{sistema1})} passing respectively 
through $\xi$ and $\xi'$ at $t=\tau$. If \textnormal{(\ref{schema0})} is verified, then it follows that
\begin{equation}
\label{schema1}
|x(t,\tau,\xi')-x(t,\tau,\xi)|\leq |\xi-\xi'|e^{p_{1}|t-\tau|}
\end{equation}
where $p_{1}$ is defined by
\begin{equation}
\label{seuil}
p_{1}=\eta_{1}+\frac{\eta_{2}e^{\eta_{1}\theta}}{1-v} \quad \textnormal{with} \quad  \eta_{1}=M+\ell_{1}, \quad \eta_{2}=M_{0}+\ell_{2} 
\end{equation}
and $v\in [0,1)$ is defined by \textnormal{(\ref{schema0})}.
\end{lemma}

\begin{proof}
Without loss of generality, we will assume that $t>\tau$, the case corresponding to $t<\tau$ can be proved similary and is left to the reader.

Firstly, let us consider the case $t_{i}<\tau<t<t_{i+1}$ for some $i\in \mathbb{Z}$, then notice that \textbf{(A1)} and \textbf{(A3)} 
imply
\begin{displaymath}
\begin{array}{rcl}
|x(t,\tau,\xi')-x(t,\tau,\xi)|&\leq & \displaystyle |\xi-\xi'|+\int_{\tau}^{t}\big\{\eta_{1}|x(s,\tau,\xi')-x(s,\tau,\xi)|\\\\
&&+\eta_{2}|x(\gamma(s),\tau,\xi')-x(\gamma(s),\tau,\xi)|\big\}\,ds.
\end{array}
\end{displaymath}

As (\ref{schema0}) implies that
$$
\int_{t_{i}}^{\zeta_{i}}\eta_{2}e^{\eta_{1}(\zeta_{i}-s)}\,ds=\frac{\eta_{2}}{\eta_{1}}\Big(e^{\eta_{1}(\zeta_{i}-t_{i})}-1\Big)\leq v<1,
$$
then Proposition \ref{LGKS} combined with
$\zeta_{i}-t_{i}\leq \theta$ for any $i\in \mathbb{Z}$ imply (\ref{schema1}) for any $t\in (\tau,t_{i+1}]$. In particular, at $t=t_{i+1}$, 
we have that
\begin{equation}
\label{schema3}
|x(t_{i+1},\tau,\xi')-x(t_{i+1},\tau,\xi)|\leq |\xi'-\xi|\exp\left(\left\{\eta_{1}+\frac{\eta_{2}e^{\eta_{1}\theta}}{1-v}\right\}(t_{i+1}-\tau)\right).
\end{equation}

Secondly, let us consider $t\in (t_{i+1},t_{i+2}]$ and notice that uniqueness of the solutions imply
\begin{equation}
\label{uni1}
x(t,t_{i+1},x(t_{i+1},\tau,\xi))=x(t,\tau,\xi),
\end{equation}
and
\begin{equation}
\label{uni2}
x(\gamma(t),t_{i+1},x(t_{i+1},\tau,\xi))=x(\gamma(t),\tau,\xi),
\end{equation}

As in the previous step, we can observe that
\begin{equation}
\label{G2}
\begin{array}{rcl}
|x(t,\tau,\xi')-x(t,\tau,\xi)|&\leq & \displaystyle |x(t_{i+1},\tau,\xi')-x(t_{i+1},\tau,\xi)|\\\\
&&+\displaystyle \int_{t_{i+1}}^{t}\big\{\eta_{1}|x(s,\tau,\xi')-x(s,\tau,\xi)|\\\\
&&+\eta_{2}|x(\gamma(s),\tau,\xi')-x(\gamma(s),\tau,\xi)|\big\}\,ds
\end{array}
\end{equation}
for any $t\in (t_{i+1},t_{i+2}]$. By applying the Gronwall's type inequality to (\ref{G2}) combined 
with (\ref{schema3}) and (\ref{uni1}), we can deduce that
\begin{displaymath}
\begin{array}{rcl}
|x(t,\tau,\xi')-x(t,\tau,\xi)| &\leq & |x(t_{i+1},\tau,\xi')-x(t_{i+1},\tau,\xi)|\exp\left(\left\{\eta_{1}+\frac{\eta_{2}e^{\eta_{1}\theta}}{1-v}\right\}(t-t_{i+1})\right)\\\\
&\leq & |\xi'-\xi|\exp\left(\left\{\eta_{1}+\frac{\eta_{2}e^{\eta_{1}\theta}}{1-v}\right\}(t-\tau)\right)
\end{array}
\end{displaymath} 
for any $t\in (t_{i+1},t_{i+2}]$ and the reader can verify that (\ref{schema1}) is valid for any $t\geq \tau$ in a recursive way.
\end{proof}

The next results are similar to the previous one and its proof is left to the reader.
\begin{lemma}
\label{CRCI2}
Let $t\mapsto y(t,\tau,\nu)$ and $t \mapsto y(t,\tau,\nu')$ be the solutions of \textnormal{(\ref{lineal-efes})} passing respectively 
through $\nu$ and $\nu'$ at $t=\tau$. If \textnormal{(\ref{schema0B})} is satisfied, then:
\begin{equation}
\label{schema1B}
|y(t,\tau,\nu')-y(t,\tau,\nu)|\leq |\nu-\nu'|e^{p_{2}|t-\tau|} \quad \textnormal{with} \quad 
p_{2}=M+\frac{M_{0}e^{M\theta}}{1-\tilde{v}}, 
\end{equation}
where $\tilde{v}\in [0,1)$ is defined by \textnormal{(\ref{schema0B})}.
\end{lemma}

\begin{lemma}
\label{CRCI3}
Let $t\mapsto x(t,\tau,\xi)$ and $t\mapsto x(t,\tau,\xi')$ (resp. $t\mapsto y(t,\tau,\nu)$ and $t \mapsto y(t,\tau,\nu')$) be the solutions 
of \textnormal{(\ref{sistema1-cor2})} (resp.\textnormal{(\ref{lineal-efes-cor2})}) passing  
through $\xi$ and $\xi'$ (resp. $\nu$ and $\nu'$) at $t=\tau$. If \textnormal{(\ref{schema0-cor2})} 
and \textnormal{(\ref{schema0B-cor2})} are satisfied, then:
\begin{equation}
\label{ef-00}
|x(t,\tau,\xi')-x(t,\tau,\xi)|\leq |\xi-\xi'|e^{\tilde{p}_{1}|t-\tau|} \quad \textnormal{with} \quad
\tilde{p}_{1}=\ell_{1}+\frac{(M_{0}+\ell_{2})e^{\ell_{1}\theta}}{1-\tilde{v}_{0}}, 
\end{equation}
and
\begin{equation}
\label{schema1B-00}
|y(t,\tau,\nu')-y(t,\tau,\nu)|\leq |\nu-\nu'|e^{\tilde{p}_{2}|t-\tau|} \quad \textnormal{with} \quad
\tilde{p}_{2}=\frac{M_{0}}{1-\tilde{u}_{0}}, 
\end{equation}
where $\tilde{v}_{0}\in [0,1)$ and $\tilde{u}_{0}\in [0,1)$ are respectively defined by \textnormal{(\ref{schema0-cor2})} and \textnormal{(\ref{schema0B-cor2})}.
\end{lemma}



\section{Proof of main results}

\subsection{Proof of Theorem \ref{intermedio}}
As stated in Remark \ref{tofin}, we only have to prove  that the maps $\xi  \mapsto H(t,\xi)$ and $\nu \mapsto L(t,\nu)$ defined in the section 4 
are uniformly continuous.

\begin{lemma}
\label{end-a1} 
The map $\xi\to H(t,\xi)=\xi+\chi(t;(t,\xi))$ is uniformly continuous for any $t$.
\end{lemma}

\begin{proof}
As the identity is uniformly continuous, we only need to prove that the map $\xi\to \chi(t;(t,\xi))$ is uniformly continuous.

Let $\xi$ and $\xi'$ be two initial conditions of (\ref{sistema1}). Notice that (\ref{homeo-00}) allows to say that
\begin{equation}
\label{decoupage}
\begin{array}{rcl}
\chi(t;(t,\xi))-\chi(t;(t,\xi'))&=&\displaystyle -\int_{-\infty}^{t}\widetilde{G}(t,s)\big\{f(s,x(s,t,\xi),x(\gamma(s),t,\xi))\\\\
&& -f(s,x(s,t,\xi'),x(\gamma(s),t,\xi'))\big\}\,ds\\\\
&&\displaystyle -\int_{t}^{\infty}\widetilde{G}(t,s)\big\{f(s,x(s,t,\xi),x(\gamma(s),t,\xi))\\\\
&&-f(s,x(s,t,\xi'),x(\gamma(s),t,\xi'))\big\}\,ds\\\\
&=&-I_{1}+I_{2}.
\end{array}
\end{equation}

Now, we divide $I_{1}$ and $I_{2}$ as follows:
\begin{displaymath}
I_{1}=\int_{-\infty}^{t-L}+\int_{t-L}^{t}=I_{11}+I_{12}
\quad
\textnormal{and}
\quad
I_{2}=\int_{t}^{t+L}+\int_{t+L}^{\infty}=I_{21}+I_{22},
\end{displaymath}
where $L$ is a positive constant.

By using \textbf{(A2)} combined with Proposition \ref{cotas-green}, we can see that the integrals 
$I_{11}$ and $I_{22}$ are always finite since 
\begin{displaymath}
|I_{11}|\leq 2K \rho^{*} \mu \int_{-\infty}^{t-L}e^{-\alpha(t-s)}\,ds=\frac{2K\mu \rho^{*}}{\alpha}e^{-\alpha L} 
\end{displaymath}
and
\begin{displaymath}
|I_{22}|\leq 2K\rho^{*} \mu \int_{t+L}^{\infty}e^{-\alpha(s-t)}\,ds = \frac{2K\mu \rho^{*}}{\alpha}e^{-\alpha L}. 
\end{displaymath}

Now, by \textbf{(A3)} and Proposition \ref{cotas-green}, we have that 
\begin{displaymath}
\begin{array}{rcl}
|I_{12}|&\leq & \displaystyle \int_{t-L}^{t}K\rho^{*} e^{-\alpha(t-s)}\ell_{1}|x(s,t,\xi)-x(s,t,\xi')|\,ds\\\\
&  & \displaystyle +\int_{t-L}^{t}K\rho^{*} e^{-\alpha(t-s)}\ell_{2}|x(\gamma(s),t,\xi)-x(\gamma(s),t,\xi')|\,ds\\\\
& \leq  & \displaystyle \int_{0}^{L}K\rho^{*} e^{-\alpha u}\ell_{1}|x(t-u,t,\xi)-x(t-u,t,\xi')|\,ds\\\\
&  & \displaystyle +\int_{0}^{L}K\rho^{*} e^{-\alpha u}\ell_{2}|x(\gamma(t-u),t,\xi)-x(\gamma(t-u),t,\xi')|\,ds.
\end{array}
\end{displaymath}

On the other hand, by Lemma \ref{CRCI}, we have that
\begin{displaymath}
0\leq |x(t-u,t,\xi)-x(t-u,t,\xi')|\leq |\xi-\xi'|e^{p_{1}L} \quad \textnormal{for any} \quad u\in [0,L].
\end{displaymath}

Similarly, by using Lemmatas \ref{EUUG} and \ref{CRCI}, we have that
\begin{displaymath}
0\leq |x(\gamma(t-u),t,\xi)-x(\gamma(t-u),t,\xi')|\leq |\xi-\xi'|e^{p_{1}(\theta+L)} \quad \textnormal{for any} \quad u\in [0,L].
\end{displaymath}

The reader can deduce that the inequalities above implies
\begin{equation}
\label{estimacion1}
|I_{12}|\leq D|\xi-\xi'| \quad \textnormal{with} \quad D=\frac{K\rho^{*} e^{p_{1}L}}{\alpha}(1-e^{-\alpha L})(\ell_{1}+\ell_{2}e^{p_{1}\theta}).
\end{equation}

Analogously, we can deduce that 
\begin{equation}
\label{estimacion2}
|I_{21}|\leq D|\xi-\xi'|.
\end{equation}

For any $\varepsilon>0$, we can choose 
\begin{displaymath}
L \geq \displaystyle \frac{1}{\alpha}\ln\big(\frac{8K\mu\rho^{*}}{\alpha \varepsilon}\big),
\end{displaymath}
which implies that
$|I_{11}|+|I_{22}|<\varepsilon/2$. By using this fact combined with (\ref{estimacion1})--(\ref{estimacion2}), we obtain 
that
\begin{displaymath}
\forall \varepsilon>0 \hspace{0.2cm}\exists \delta=\frac{\varepsilon}{4D}>0 \quad \textnormal{such that} \quad
|\xi-\xi'|<\delta \Rightarrow |\chi(t;(t,\xi))-\chi(t;(t,\xi'))|<\varepsilon
\end{displaymath}
and the uniform continuity follows. 
\end{proof}

\begin{lemma}
\label{end-a2}
The map $\nu\mapsto L(t,\nu)=\nu+\vartheta(t;(t,\nu))$ is uniformly continuous for any $t$. 
\end{lemma}
\begin{proof}
We only need to prove that the map $\nu\mapsto \vartheta(t;(t,\nu))$ is uniformly continuous. In order to prove that,
let $\nu$ and $\nu'$ be two initial conditions of (\ref{lineal-efes}) and define
$$
\Delta=\vartheta(t;(t,\nu))-\vartheta(t;(t,\nu')).
$$

By using (\ref{fijo}), we can see that $\Delta$ can be written as follows:
\begin{equation}
\label{decoupage2}
\begin{array}{rl}
\Delta=\\\\
&\displaystyle \int_{-\infty}^{t}\hspace{-0.2cm}\widetilde{G}(t,s)\big\{f(s,y(s,t,\nu)+\vartheta(s;(t,\nu)),
y(\gamma(s),t,\nu)+\vartheta(\gamma(s);(t,\nu)))\\\\
& -f(s,y(s,t,\nu')+\vartheta(s;(t,\nu')),y(\gamma(s),t,\nu')+\vartheta(\gamma(s);(t,\nu')))\big\}\,ds+\\\\
&\displaystyle \int_{t}^{\infty}\hspace{-0.2cm}\widetilde{G}(t,s)\big\{f(s,y(s,t,\nu)+\vartheta(s;(t,\nu)),y(\gamma(s),t,\nu)+\vartheta(\gamma(s);(t,\nu)))\\\\
&-f(s,y(s,t,\nu')+\vartheta(s;(t,\nu')),y(\gamma(s),t,\nu')+\vartheta(\gamma(s);(t,\nu')))\big\}\,ds\\\\
=&J_{1}+J_{2}.
\end{array}
\end{equation}

As before, we divide $J_{1}$ and $J_{2}$ as follows:
\begin{displaymath}
J_{1}=\int_{-\infty}^{t-\tilde{L}}+\int_{t-\tilde{L}}^{t}=J_{11}+J_{12}, \quad
J_{2}=\int_{t}^{t+\tilde{L}}+\int_{t+\tilde{L}}^{\infty}=J_{21}+J_{22}.
\end{displaymath}

By \textbf{(A2)} and Proposition \ref{cotas-green}, it is straightforward to verify that
\begin{displaymath}
|J_{11}|\leq \frac{2K\rho^{*} \mu}{\alpha}e^{-\alpha \tilde{L}} \quad \textnormal{and} \quad
|J_{22}|\leq \frac{2K\rho^{*} \mu}{\alpha}e^{-\alpha \tilde{L}}. 
\end{displaymath}

Let us define
\begin{equation}
\label{Ninfty}
||\vartheta(\cdot;(t,\nu))-\vartheta(\cdot;(t,\nu'))||_{\infty}=\sup\limits_{s\in (-\infty,\infty)}|\vartheta(s;(t,\nu))-\vartheta(s;(t,\nu'))|,
\end{equation}
and notice that \textbf{(A3)} and Proposition \ref{cotas-green} implies:
\begin{displaymath}
\begin{array}{rcl}
|J_{12}|&\leq & \displaystyle \frac{K\rho^{*}}{\alpha}(\ell_{1}+\ell_{2})||\vartheta(\cdot;(t,\nu))-\vartheta(\cdot;(t,\nu'))||_{\infty}\\\\
        & &\displaystyle +K\rho^{*} \ell_{1}\int_{t-\tilde{L}}^{t}e^{-\alpha(t-s)}|y(s,t,\nu)-y(s,t,\nu')|\,ds\\\\
        & &\displaystyle +K\rho^{*} \ell_{2}\int_{t-\tilde{L}}^{t}e^{-\alpha(t-s)}|y(\gamma(s),t,\nu)-y(\gamma(s),t,\nu')|\,ds\\\\
        &\leq & \displaystyle \frac{K\rho^{*}}{\alpha}(\ell_{1}+\ell_{2})||\vartheta(\cdot;(t,\nu))-\vartheta(\cdot;(t,\nu'))||_{\infty}\\\\
        & &\displaystyle +K\rho^{*} \ell_{1}\int_{0}^{L}e^{-\alpha u}|y(t-u,t,\nu)-y(t-u,t,\nu')|\,ds\\\\
        & &\displaystyle +K\rho^{*} \ell_{2}\int_{0}^{L}e^{-\alpha u}|y(\gamma(t-u),t,\nu)-y(\gamma(t-u),t,\nu')|\,ds.\\\\
\end{array}
\end{displaymath}

By using Lemma \ref{CRCI2}, we know that
\begin{displaymath}
|y(t-u,t,\nu)-y(t-u,t,\nu')|\leq |\nu-\nu'|e^{p_{2}\tilde{L}} \quad \textnormal{for any} \quad u\in [0,\tilde{L}]
\end{displaymath}
and by using again Lemmatas \ref{CRCI2} and \ref{EUUG}, we have
\begin{displaymath}
|y(\gamma(t-u),t,\nu)-y(\gamma(t-u),t,\nu')|\leq |\nu-\nu'|e^{p_{2}(\theta+\tilde{L})} \quad \textnormal{for any} \quad u\in [0,\tilde{L}]
\end{displaymath}
and the reader can deduce that 
\begin{displaymath}
|J_{12}|\leq \frac{K\rho^{*}}{\alpha}(\ell_{1}+\ell_{2})||\vartheta(\cdot;(t,\nu))-\vartheta(\cdot;(t,\nu'))||_{\infty}+\tilde{D}|\nu-\nu'| 
\end{displaymath}
with 
$$
\tilde{D}=\frac{K\rho^{*} e^{p_{2}L}}{\alpha}(1-e^{-\alpha \tilde{L}})(\ell_{1}+\ell_{2}e^{p_{2}\theta}),
$$
in addition, the following inequality can be proved in a similar way
\begin{displaymath}
|J_{21}|\leq \frac{K\rho^{*}}{\alpha}(\ell_{1}+\ell_{2})||\vartheta(\cdot;(t,\nu))-\vartheta(\cdot;(t,\nu'))||_{\infty}+\tilde{D}|\nu-\nu'|. 
\end{displaymath}

By using the inequalities stated above combined with (\ref{FPT}), he have
\begin{displaymath}
\begin{array}{rcl}
|\vartheta(t;(t,\nu))-\vartheta(t;(t,\nu'))|&\leq& |J_{11}|+|J_{22}|+|J_{12}|+|J_{21}|\\\\
&\leq & \displaystyle \frac{4K\rho^{*} \mu}{\alpha}e^{-\alpha \tilde{L}}+2\tilde{D}|\nu-\nu'|\\\\
& &\displaystyle  + \frac{2K\rho^{*}}{\alpha}(\ell_{1}+\ell_{2})||\vartheta(\cdot;(t,\nu))-\vartheta(\cdot;(t,\nu'))||_{\infty},
\end{array}
\end{displaymath}
and we obtain 
\begin{displaymath}
\begin{array}{rcl}
|\vartheta(t;(t,\nu))-\vartheta(t;(t,\nu'))|&\leq & \displaystyle \frac{4K\rho^{*} \mu e^{-\alpha \tilde{L}}}{\alpha(1-\Gamma^{*})}+\frac{2\tilde{D}}{1-\Gamma^{*}}|\nu-\nu'|.
\end{array}
\end{displaymath}
with $\Gamma^{*}$ defined by
\begin{displaymath}
\Gamma^{*}=\frac{2K\rho^{*}}{\alpha}(\ell_{1}+\ell_{2})<1.
\end{displaymath}

Finally, for any $\varepsilon>0$, we can choose 
\begin{displaymath}
\tilde{L} \geq \displaystyle \frac{1}{\alpha}\ln\Big(\frac{8K\mu\rho^{*}}{\alpha \varepsilon (1-\Gamma^{*})}\Big),
\end{displaymath}
which implies that
$\frac{4K\rho^{*} \mu}{\alpha(1-\Gamma^{*})}e^{-\alpha \tilde{L}} <\varepsilon/2$. By using this fact, we obtain 
that
\begin{displaymath}
\forall \varepsilon>0 \hspace{0.2cm}\exists \delta=\frac{\varepsilon}{4\tilde{D}(1-\Gamma^{*})}>0 \quad \textnormal{such that} \quad
|\nu-\nu'|<\delta \Rightarrow |\vartheta(t;(t,\nu))-\vartheta(t;(t,\nu'))|<\varepsilon
\end{displaymath}
and the uniform continuity follows. 
\end{proof}

\subsection{Proof of Theorem \ref{intermedio2}}
As before, we  only have to prove that the maps $\xi  \mapsto H(t,\xi)$ and $\nu \mapsto L(t,\nu)$ defined in the section 4 are H\"older continous.
 
\begin{lemma}
For any couple $\xi$ and $\xi'$ such that $|\xi-\xi'|<1$, there exists $C_{1}>1$ such that the map 
$\xi\to H(t,\xi)=\xi+\chi(t;(t,\xi))$ verifies
$$
|H(t,\xi)-H(t,\xi')|\leq C_{1}|\xi-\xi'|^{\frac{\alpha}{p_{1}}} \quad \textnormal{for any} \quad t\in \mathbb{R},
$$
with $p_{1}$ defined by \textnormal{(\ref{seuil})}.
\end{lemma}

\begin{proof}
As before, we only need to prove that the map $\xi\mapsto \chi(t;(t,\xi))$ is uniformly continuous. Now, we use the
the identity 
\begin{displaymath}
\begin{array}{rcl}
\chi(t;(t,\xi))-\chi(t;(t,\xi'))&=&-I_{1}+I_{2},
\end{array}
\end{displaymath}
described by (\ref{decoupage}). Nevertheless, this time we consider the intervals $I_{1}$ and $I_{2}$:
\begin{displaymath}
I_{1}=\int_{-\infty}^{t-T}+\int_{t-T}^{t}=I_{11}+I_{12}, \quad
I_{2}=\int_{t}^{t+T}+\int_{t+T}^{\infty}=I_{21}+I_{22},
\end{displaymath}
where
\begin{equation}
\label{threshold}
T=\frac{1}{p_{1}}\ln\Big(\frac{1}{|\xi-\xi'|}\Big).
\end{equation}

The reader can easily verify that
\begin{equation}
\label{T2}
e^{-\alpha T}=|\xi-\xi'|^{\frac{\alpha}{p_{1}}} \quad \textnormal{and} \quad e^{p_{1}T}=|\xi-\xi'|^{-1}, 
\end{equation}
which combined with (\ref{cota-imp}) implies that
\begin{displaymath}
\displaystyle |I_{11}|\leq \frac{2\mu K\rho^{*}}{\alpha}|\xi-\xi'|^{\frac{\alpha}{p_{1}}} \quad \textnormal{and} \quad
|I_{22}|\leq \frac{2\mu K\rho^{*}}{\alpha}|\xi-\xi'|^{\frac{\alpha}{p_{1}}}
\end{displaymath}

By using \textbf{(A3)}, Proposition \ref{cotas-green} and Lemma \ref{CRCI}, we have that 
\begin{displaymath}
\begin{array}{rcl}
|I_{21}|&\leq & \displaystyle \int_{t}^{t+T}K\rho^{*} e^{-\alpha(s-t)}\ell_{1}|x(s,t,\xi)-x(s,t,\xi')|\,ds\\\\
&  & \displaystyle +\int_{t}^{t+T}K\rho^{*} e^{-\alpha(s-t)}\ell_{2}|x(\gamma(s),t,\xi)-x(\gamma(s),t,\xi')|\,ds\\\\
& \leq & \displaystyle |\xi-\xi'|K\rho^{*}\ell_{1}  \int_{t}^{t+T}e^{(p_{1}-\alpha)(s-t)}\,ds\\\\
& & +\displaystyle |\xi-\xi'|K\rho^{*}\ell_{2}  \int_{t}^{t+T}e^{-\alpha(s-t)}e^{p_{1}|\gamma(s)-t|}\,ds.
\end{array}
\end{displaymath}

By using Lemma \ref{EUUG}, we can see that
\begin{displaymath}
\begin{array}{rcl}
|I_{21}|&\leq &\displaystyle  \big\{\ell_{1}+\ell_{2}e^{p_{1}\theta}\big\}|\xi-\xi'|K\rho^{*} \int_{t}^{t+T}e^{(p_{1}-\alpha)(s-t)}\,ds.
\end{array}
\end{displaymath}

Now, by (\ref{alfa}), we have that $p_{1}>\alpha$. By using this fact combined with (\ref{T2}), we obtain:
\begin{displaymath}
\begin{array}{rcl}
|I_{21}|\leq \displaystyle \frac{K\rho^{*}}{p_{1}-\alpha} \big\{\ell_{1}+\ell_{2}e^{p_{1}\theta}\big\}|\xi-\xi'|^{\frac{\alpha}{p_{1}}}.
\end{array}
\end{displaymath}

A similar estimation can be obtained for $I_{12}$:
\begin{displaymath}
\begin{array}{rcl}
|I_{12}|\leq \displaystyle \frac{K\rho^{*}}{p_{1}-\alpha} \big\{\ell_{1}+\ell_{2}e^{p_{1}\theta}\big\}|\xi-\xi'|^{\frac{\alpha}{p_{1}}}.
\end{array}
\end{displaymath}

Finally, as $\alpha<p_{1}$ and $|\xi-\xi'|<1$, we can conclude that 
\begin{displaymath}
\begin{array}{rcl}
|H(t,\xi)-H(t,\xi')|&\leq & |\xi-\xi'|+|\chi(t;(t,\xi))-\chi(t;(t,\xi'))|\\\\
&\leq  & \displaystyle \Big(1+\frac{2K\rho^{*}}{p_{1}-\alpha} \big\{\ell_{1}+\ell_{2}e^{p_{1}\theta}\big\}+\frac{4\mu K\rho^{*}}{\alpha}\Big)|\xi-\xi|^{\frac{\alpha}{p_{1}}}.
\end{array}
\end{displaymath}
\end{proof}


\begin{lemma}
For any couple $\nu$ and $\nu'$ such that $|\nu-\nu'|<1$, there exists $D_{1}>1$ such that 
the map $\xi\to L(t,\xi)=\xi+\vartheta(t;(t,\nu))$ verifies
\begin{displaymath}
|L(t,\nu)-L(t,\nu')|\leq D_{1}|\nu-\nu'|^{\frac{\alpha}{p_{2}}}.
\end{displaymath}

\end{lemma}
\begin{proof}
As in the previous proof, we will start by studying the map $\nu \to \vartheta(t;(t,\nu))$. Let us
recall the identity 
\begin{displaymath}
\begin{array}{rcl}
|\vartheta(t;(t,\nu))-\vartheta(t;(t,\nu'))|&=&J_{1}+J_{2},
\end{array}
\end{displaymath}
described by (\ref{decoupage2}). As before, we divide $J_{1}$ and $J_{2}$ as follows:
\begin{displaymath}
J_{1}=\int_{-\infty}^{t-\tilde{T}}+\int_{t-\tilde{T}}^{t}=J_{11}+J_{12}, \quad
J_{2}=\int_{t}^{t+\tilde{T}}+\int_{t+\tilde{T}}^{\infty}=J_{21}+J_{22},
\end{displaymath}
with $\tilde{T}$ defined by
\begin{displaymath}
\tilde{T}=\frac{1}{p_{2}}\ln\Big(\frac{1}{|\nu-\nu'|}\Big).
\end{displaymath}

The inequalities
\begin{displaymath}
|J_{11}|\leq \frac{2\mu K}{\alpha}|\nu-\nu'|^{\frac{\alpha}{p_{2}}} \quad \textnormal{and}\quad
|J_{22}|\leq \frac{2\mu K}{\alpha}|\nu-\nu'|^{\frac{\alpha}{p_{2}}}
\end{displaymath}
can be proved analogously as before.

By using \textbf{(A3)} combined with Proposition \ref{cotas-green} and Lemma \ref{CRCI2}, we can deduce that
\begin{displaymath}
\begin{array}{rcl}
|J_{12}|&\leq & \displaystyle \int_{t-\tilde{T}}^{t}K\rho^{*} e^{-\alpha(t-s)}\ell_{1}|y(s,t,\nu)-y(s,t,\nu')|\,ds\\\\
&  & \displaystyle +\int_{t-\tilde{T}}^{t}K\rho^{*} e^{-\alpha(t-s)}\ell_{2}|y(\gamma(s),t,\nu)-y(\gamma(s),t,\nu')|\,ds\\\\
&  & \displaystyle +\int_{t-\tilde{T}}^{t}K\rho^{*} e^{-\alpha(t-s)}\ell_{1}|\vartheta(s;(t,\nu))-\vartheta(s;(t,\nu'))|\,ds\\\\
&  & \displaystyle +\int_{t-\tilde{T}}^{t}K\rho^{*} e^{-\alpha(t-s)}\ell_{2}|\vartheta(\gamma(s);(t,\nu))-\vartheta(\gamma(s);(t,\nu'))|\,ds\\\\
& \leq & \displaystyle \frac{K\rho^{*}}{p_{2}-\alpha} \Big\{\ell_{1}+\ell_{2}e^{p_{2}\theta}\Big\}|\nu-\nu'|^{\frac{\alpha}{p_{2}}}+\frac{K\rho^{*}}{\alpha}(\ell_{1}+\ell_{2})||\vartheta(\cdot;(t,\nu))-\vartheta(\cdot;(t,\nu'))||_{\infty},
\end{array}
\end{displaymath}
where $p_{2}>\alpha$ is consequence of (\ref{alfa}). Let us recall that $||\vartheta(\cdot;(t,\nu))-\vartheta(\cdot;(t,\nu'))||_{\infty}$ is defined by (\ref{Ninfty}).

Similarly, we can deduce that
\begin{displaymath}
\begin{array}{rcl}
|J_{21}|&\leq & \displaystyle \frac{K\rho^{*}}{p_{2}-\alpha} \Big\{\ell_{1}+\ell_{2}e^{p_{2}\theta}\Big\}|\nu-\nu'|^{\frac{\alpha}{p_{2}}}+
\frac{K\rho^{*}}{\alpha}(\ell_{1}+\ell_{2})||\vartheta(\cdot;(t,\nu))-\vartheta(\cdot;(t,\nu'))||.
\end{array}
\end{displaymath}

Then, we obtain
\begin{displaymath}
\begin{array}{rcl}
|\vartheta(t;(t,\nu))-\vartheta(t;(t,\nu'))| &\leq & |J_{11}|+|J_{12}|+|J_{21}|+|J_{22}|\\\\
& \leq &\displaystyle \frac{2K\rho^{*}}{p_{2}-\alpha} \Big\{\ell_{1}+\ell_{2}e^{p_{2}\theta}\Big\}|\nu-\nu'|^{\frac{\alpha}{p_{2}}}+
\frac{4\mu K}{\alpha}|\nu-\nu'|^{\frac{\alpha}{p_{2}}}\\\\
& &+\displaystyle\frac{2K\rho^{*}}{\alpha}(\ell_{1}+\ell_{2})||\vartheta(\cdot;(t,\nu))-\vartheta(\cdot;(t,\nu'))||_{\infty},
\end{array}
\end{displaymath}
which implies that
\begin{displaymath}
\begin{array}{rcl}
||\vartheta(\cdot;(t,\nu))-\vartheta(\cdot;(t,\nu'))||_{\infty} &\leq & \displaystyle \Big(\frac{2K\rho^{*}}{p_{2}-\alpha} \Big\{\ell_{1}+\ell_{2}e^{p_{2}\theta}\Big\}+
\frac{4\mu K}{\alpha}\Big)|\nu-\nu'|^{\frac{\alpha}{p_{2}}}\\\\
& &+\displaystyle\frac{2K\rho^{*}}{\alpha}(\ell_{1}+\ell_{2})||\vartheta(\cdot;(t,\nu))-\vartheta(\cdot;(t,\nu'))||_{\infty}.
\end{array}
\end{displaymath}

Now, by using (\ref{FPT}), we conclude that
\begin{displaymath}
|\vartheta(t;(t,\nu))-\vartheta(t;(t,\nu'))|\leq \big(1-\Gamma^{*}\big)^{-1}\Big( \frac{2K\rho^{*}}{p_{2}-\alpha} \Big\{\ell_{1}+\ell_{2}e^{p_{2}\theta}\Big\}+\frac{4\mu K}{\alpha}\Big)|\nu-\nu'|^{\frac{\alpha}{p_{2}}},
\end{displaymath}
and the inequality $p_{2}>\alpha$ combined with $|\nu-\nu'|<1$ allows to deduce
\begin{displaymath}
\begin{array}{rcl}
|L(t,\nu)-L(t,\nu')|&\leq & |\nu-\nu'|+|\vartheta(t;(t,\nu))-\vartheta(t;(t,\nu'))|\\\\
&\leq& \displaystyle \Bigg(1+\frac{\displaystyle \frac{2K\rho^{*}}{p_{2}-\alpha} \big(\ell_{1}+\ell_{2}e^{p_{2}\theta}\big)+\frac{4\mu K}{\alpha}}
{1-\Gamma^{*}}\Bigg)|\nu-\nu'|^{\frac{\alpha}{p_{2}}}
\end{array}
\end{displaymath}
and the result follows.

\end{proof}




-



\end{document}